\documentclass{amsart}
\usepackage{amssymb}
\usepackage{amsmath}
\usepackage{amsthm}
\usepackage[all]{xy}
\usepackage{eucal}

\SelectTips{cm}{}
\theoremstyle{plain}
\newtheorem{thm}{Theorem}[section]
\newtheorem{cor}[thm]{Corollary}
\newtheorem{lem}[thm]{Lemma}
\newtheorem{prop}[thm]{Proposition}
\theoremstyle{definition}
\newtheorem{defn}[thm]{Definition}
\theoremstyle{remark}
\newtheorem{rem}[thm]{Remark}

\newcommand{\Z}{{\mathbb Z}}
\newcommand{\Q}{{\mathbb Q}}
\newcommand{\C}{\mathcal{C}}
\newcommand{\M}{\mathcal{M}}
\newcommand{\E}{\mathcal{E}}

\newcommand{\sSets}{{\rm sSets}}
\newcommand{\Ass}{{\mathcal{A}\rm ss}}
\newcommand{\Com}{{\mathcal{C}\rm om}}
\newcommand{\Ab}{{\mathcal{A}\rm b}}
\newcommand{\Mod}{{\mathcal{M}\rm od}}
\newcommand{\LMod}{{\mathcal{L}\rm Mod}}
\newcommand{\RMod}{{\mathcal{R}\rm Mod}}
\newcommand{\BMod}{{\mathcal{B}\rm Mod}}
\newcommand{\Mor}{{\mathcal{M}\rm or}}
\newcommand{\Ho}{{\mathrm{Ho}}}
\newcommand{\Hom}{\mathop{\textrm{\rm Hom}}}
\newcommand{\Map}{\mathop{\textrm{\rm Map}}}
\newcommand{\map}{\mathop{\textrm{\rm map}}}
\newcommand{\End}{\mathop{\textrm{\rm End}}}

\numberwithin{equation}{section}

\begin{document}
\title{Localization of algebras over coloured operads}
\author[C. Casacuberta]{Carles Casacuberta}
\thanks{The first and second named authors were supported by the Spanish Ministry of Education
and Science under grants MTM2004\nobreakdash-03629, MTM2007\nobreakdash-63277,
and EX2005\nobreakdash-0521}
\address{Departament d'\`Algebra i Geometria, Universitat de Barcelona,
Gran Via de les Corts Catalanes, 585, 08007 Barcelona, Spain}
\email{carles.casacuberta@ub.edu}
\author[J. J. Guti\'errez]{Javier J. Guti\'errez}
\address{Centre de Recerca Matem\`atica, Apartat 50, 08193 Bellaterra, Spain}
\email{jgutierrez@crm.cat}
\author[I. Moerdijk]{\\ Ieke Moerdijk}
\address{Mathematisch Instituut, Postbus 80.010, 3508 TA Utrecht, The Netherlands}
\email{i.moerdijk@uu.nl}
\author[R. M. Vogt]{Rainer M. Vogt}
\address{Universit\"at Osnabr\"uck, Fachbereich Mathematik/Informatik, Albrechtstr.\ 28,
49069 Osnabr\"uck, Germany}
\email{rainer@mathematik.uni\nobreakdash-osnabrueck.de}


\begin{abstract}
We give sufficient conditions for homotopical localization functors to preserve algebras over
coloured operads in monoidal model categories. Our approach encompasses a number of previous results
about preservation of structures under localizations, such as loop spaces or infinite loop spaces,
and provides new results of the same kind.
For instance, under suitable assumptions, homotopical localizations preserve ring spectra
(in the strict sense, not only up to homotopy), modules over ring spectra, and algebras over
commutative ring spectra, as well as ring maps, module maps, and algebra maps.
It is principally the treatment of module spectra
and their maps that led us to the use of coloured operads (also called enriched multicategories)
in this context.
\end{abstract}
\maketitle

\section*{Introduction}
\label{introduction}
A remarkable property of localizations in homotopy theory is the fact that they preserve many kinds of algebraic structures.
That is, if a space or a spectrum $X$ is equipped with some structure and $L$ is a homotopical localization functor
(such as, for example, localization at a set of primes, localization with respect to a homology theory, or a Postnikov section),
very often $LX$ admits the same structure as~$X$, in fact in a unique way (up to homotopy) if we impose the condition
that the localization map $X\longrightarrow LX$ be compatible with the structure.

For instance, it is known that $f$\nobreakdash-localizations in the sense of Bousfield \cite{Bou94}, \cite{Bou96} and Farjoun \cite{Far96}
preserve the classes of homotopy associative $H$\nobreakdash-spaces, loop spaces, and infinite loop spaces, among others.
Such $f$\nobreakdash-localizations also preserve GEMs (i.e., products of Eilenberg\nobreakdash--Mac Lane spaces), as explained in \cite{Far96},
as well as other classes of spaces defined by means of algebraic theories \cite{Bad02}.

In the stable homotopy category, $f$\nobreakdash-localizations
that commute with the suspension operator preserve homotopy ring spectra and homotopy module spectra \cite{CG05}.
Furthermore, if a homotopy ring spectrum $R$ is connective, then the class of homotopy modules over $R$
is preserved by all $f$\nobreakdash-localizations, not necessarily commuting with suspension; see \cite{Bou99}, \cite{CG05}.
As a consequence, the class of stable GEMs is preserved by all $f$\nobreakdash-localizations, since stable
GEMs are precisely homotopy modules over the integral Eilenberg\nobreakdash--Mac Lane spectrum~$H\Z$.
(Note that either some connectivity condition or the assumption that the given localization commutes with suspension is necessary, 
since Postnikov sections of Morava $K$\nobreakdash-theory spectra $K(n)$ need neither be
homotopy ring spectra nor homotopy modules over $K(n)$, as observed by Rudyak in~\cite{Rud98}.)

A common feature of these examples is that they can be described in terms of algebras over operads or,
in some cases, algebras over coloured operads.
Coloured operads first appeared in the book of Boardman and Vogt \cite{BV73} on
homotopy invariant algebraic structures on topological spaces. They can be viewed as
multicategories \cite{Lam69} enriched over a symmetric monoidal category and equipped with a symmetric group action.
Under suitable conditions, coloured operads carry a model structure (see \cite{BM03}, \cite{BM07}),
which enables one to speak in a systematic way, for a coloured operad $P$, about
homotopy $P$\nobreakdash-algebras as being $P_{\infty}$\nobreakdash-algebras, for a cofibrant resolution $P_{\infty}\longrightarrow P$.

In this article, we study the preservation of classes of algebras over coloured operads under the effect
of localizations in closed symmetric monoidal categories, and under the effect of homotopical localizations
in simplicial or topological monoidal model categories (see \cite{Qui67} or \cite{Hov99} for background about model categories).

We prove the following.
Let $C$ be any set and $P$ a \emph{cofibrant} $C$\nobreakdash-coloured operad
in the category of simplicial sets (or compactly generated spaces) acting on a simplicial
(or topological) monoidal model category~$\M$. Let $L$ be a homotopical
localization functor on $\M$ whose class of equivalences is closed under tensor products.
If ${\bf X}=(X(c))_{c\in C}$ is a $P$\nobreakdash-algebra with $X(c)$ cofibrant for all $c$,
then $L{\bf X}=(LX(c))_{c\in C}$ admits a homotopy unique $P$\nobreakdash-algebra structure such that the
localization map ${\bf X}\longrightarrow L{\bf X}$ is a map of $P$\nobreakdash-algebras. See Theorem~\ref{mainthm}
below for a more general variant of this statement.

As an example of this result, we mention the following fact, which
has been known in slightly more restrictive forms for several years.
Let $L$ be a homotopical localization functor on the category of simplicial sets or compactly generated spaces.
If $X$ is a cofibrant $A_{\infty}$\nobreakdash-space, then $LX$ has a homotopy unique $A_{\infty}$\nobreakdash-space structure such that
the localization map $X\longrightarrow LX$ is a map of $A_{\infty}$\nobreakdash-spaces.
The same statement is true for $E_{\infty}$\nobreakdash-spaces. (Here and throughout we denote
by $A_{\infty}$ a cofibrant replacement of the associative operad $\Ass$, and by $E_{\infty}$
a cofibrant replacement of the commutative operad $\Com$.)
Since any $A_{\infty}$\nobreakdash-space is weakly equivalent to a topological monoid,
this result implies that homotopical localizations preserve topological monoids up to homotopy.
Moreover, since nontrivial homotopical localizations induce bijections on connected components,
they also preserve loop spaces (that is, group\nobreakdash-like $A_{\infty}$\nobreakdash-spaces) up to homotopy.

In the stable case, our result is illustrated as follows.
Let $L$ be a homotopical localization functor on the category of symmetric spectra \cite{HSS00}.
Let $M$ be an $A_{\infty}$\nobreakdash-module over an $A_{\infty}$\nobreakdash-ring $R$, where both $M$
and $R$ are assumed to be cofibrant as spectra.
Firstly, if $R$ is connective or the functor $L$ commutes with suspension, then $LM$ has a homotopy unique
$A_{\infty}$\nobreakdash-module structure over $R$ such that the localization map $M\longrightarrow LM$ is a map
of $A_{\infty}$\nobreakdash-modules.
Secondly, if $L$ commutes with suspension, then $LR$ has a homotopy unique
$A_{\infty}$\nobreakdash-ring structure such that the localization map $R\longrightarrow LR$ is a map of $A_{\infty}$\nobreakdash-rings,
and $LM$ then admits a homotopy unique $A_{\infty}$\nobreakdash-module structure over $LR$ extending the $A_{\infty}$\nobreakdash-module
structure over~$R$. If $L$ does not commute with suspension, then the same holds if we
assume that $R$, $LR$, and at least one of $M$ and $LM$ are connective.
The same statements are true if $A_{\infty}$ is replaced by $E_{\infty}$.
(We emphasize that $E_{\infty}$\nobreakdash-algebras are weakly equivalent to commutative monoids in the category of
symmetric spectra, according to \cite{GH04} or~\cite {EM06}, but not in the category of
simplicial sets ---since infinite loop spaces need not be GEMs--- or in other monoidal model categories.)

We subsequently deduce the preservation of strict ring structures (also commutative)
and strict module structures under homotopical localizations in the category of symmetric spectra.
In~fact, we show in Section~\ref{rectifications} that each localization of a ring morphism between
strict ring spectra is naturally weakly equivalent (in the category of maps between spectra)
to a ring morphism, and similarly for $R$\nobreakdash-modules,
under appropriate connectivity assumptions.
For this, we view such morphisms as algebras over coloured operads,
as in \cite[2.9]{MSS02} or~\cite[1.5.3]{BM07}, and use the corresponding functorial rectifications
from~\cite{EM06}. (Rectification of algebras has been studied in the context of categories
with cartesian product by Badzioch \cite{Bad02} and Bergner \cite{Ber06}.)

As we show in the last section, it is also true that, for every commutative
ring spectrum~$R$, the class of $R$\nobreakdash-algebras is preserved under homotopical localizations commuting
with suspension. If the localization is homological, then the localized $R$\nobreakdash-algebra structures
coincide up to homotopy with those obtained in \cite[Theorem VIII.2.1]{EKMM97}.
In another direction,
it was proved in \cite{Laz01} and \cite{DS06} that Postnikov pieces of connective $R$\nobreakdash-algebras
admit compatible $R$\nobreakdash-algebra structures, provided that $R$ is itself connective. Our approach also
yields this as a special case.

When we refer to the model category of symmetric spectra, we will understand
it in the sense of the positive stable model structure, which was discussed in
\cite{MMSS01}, \cite{Sch01}, or \cite{Shi04}.
Likewise, when we speak of compactly generated spaces we mean
$k$\nobreakdash-spaces without any separation condition, as in \cite{Vog71}, equipped with Quillen's
model category structure (given by weak homotopy equivalences, Serre fibrations, and the corresponding cofibrations).
This model structure has the advantage of being cofibrantly generated, which ensures
the validity of certain results that could fail to hold otherwise, mainly the existence of
an adequate model category structure on the category of coloured operads over a fixed set of colours.
For certain purposes, however, it is more convenient to consider the $k$\nobreakdash-space version of the Str{\o}m model
category structure \cite{Str72}, with genuine homotopy equivalences,
Hurewicz fibrations and closed cofibrations. In this case, all spaces are fibrant and cofibrant.
Although this model category is not known to be cofibrantly generated, one can still speak about operads being cofibrant,
in the sense of having a left lifting property with respect to morphisms that induce trivial fibrations
in the underlying category. It was proved in \cite{Vog03} that the 
$W$\nobreakdash-construction yields operads that are cofibrant in this sense.

\medskip

\noindent
\textit{Acknowledgements}. The plausibility of an interaction between localizations and operads
was seen by several people shortly after the development of $f$\nobreakdash-localizations, at a moment
where the technical machinery for a broad statement and proof was not yet fully available.
Some of us had discussions on this topic with W.~Chach\'olski, G.~Granja, B.~Richter, B.~Shipley, and J.~H.~Smith.
Part of this work was done while the second\nobreakdash-named author was visiting Utrecht University.

\section{Coloured operads and their algebras}
\label{background}
In the first two sections, $\E$ will denote a cocomplete closed symmetric monoidal category
with tensor product $\otimes$, unit $I$, and internal hom functor $\Hom_{\E}(-,-)$.
We denote by $0$ the initial object of $\E$, that is, a colimit of the empty diagram.

Let $\Sigma_n$ denote the symmetric group on $n$ elements (which is meant to be the trivial group
if $n=0$ and $n=1$), and let $C$ be a set, whose elements will be called \emph{colours}.
A \emph{$C$\nobreakdash-coloured collection} $K$ in $\E$ consists of a set of objects $K(c_1,\ldots, c_n; c)$ in $\E$
for every $n\ge 0$ and each $(n+1)$\nobreakdash-tuple of colours $(c_1,\ldots, c_n; c)$, equipped with a right
action of $\Sigma_n$ by means of maps
\[
\sigma^*\colon K(c_1,\ldots,c_n; c)\longrightarrow K(c_{\sigma(1)},\ldots,c_{\sigma(n)}; c)
\]
where $\sigma\in\Sigma_n$.
The objects corresponding to $n=0$ are denoted by $K(\; ; c)$.

A \emph{morphism} $F\colon K\longrightarrow K'$ of $C$\nobreakdash-coloured collections is a family of maps
\[
K(c_1,\ldots, c_n; c)\longrightarrow K'(c_1,\ldots, c_n; c)
\]
in $\E$, ranging over $n\ge 0$ and all $(n+1)$\nobreakdash-tuples of colours $(c_1,\ldots, c_n;c)$,
compatible with the action of the symmetric groups.
We denote by ${\rm Coll}_C(\E)$ the category of $C$\nobreakdash-coloured collections in $\E$ with their morphisms,
for a fixed set of colours $C$.

A \emph{$C$\nobreakdash-coloured operad} $P$ in $\mathcal{E}$ is a $C$\nobreakdash-coloured collection equipped with
a \emph{unit} map $I\longrightarrow P(c;c)$ for each $c$ in $C$ and,
for every $(n+1)$\nobreakdash-tuple of colours $(c_1,\ldots, c_n;c)$ and $n$ given tuples
\[
(a_{1,1},\ldots, a_{1,k_1}; c_1),\ldots, (a_{n,1},\ldots, a_{n,k_n}; c_n),
\]
a \emph{composition product} map
\begin{gather}\notag
P(c_1,\ldots, c_n;c)\otimes P(a_{1,1},\ldots, a_{1,k_1};c_1)\otimes\cdots\otimes P(a_{n,1},\ldots, a_{n,k_n};c_n)
\\ \longrightarrow P(a_{1,1},\ldots,a_{1,k_1},a_{2,1},\ldots,a_{2,k_2},\ldots,a_{n,1},\ldots,a_{n,k_n};c), \notag
\end{gather}
compatible with the action of the symmetric groups and subject to associativity
and unitary compatibility relations; see, e.g., \cite[\S 2]{EM06} for a depiction of the diagrams involved.
Thus, we may view a $C$\nobreakdash-coloured operad $P$ as a \textit{multicategory} enriched over $\E$, where
the hom objects $P(c_1,\ldots,c_n;c)$ have $n$ inputs and one output.

A morphism of $C$\nobreakdash-coloured operads is a morphism of the underlying $C$\nobreakdash-coloured collections
that is compatible with the unit maps and the composition product maps.
The category of $C$\nobreakdash-coloured operads in $\E$ will be denoted by ${\rm Oper}_C(\E)$.
As shown in \cite[Appendix]{BM07}, the category of $C$\nobreakdash-coloured collections admits a monoidal structure
in which the monoids are precisely the $C$\nobreakdash-coloured operads.

We note, for later use, that the forgetful functor ${\rm Oper}_C(\E)\longrightarrow {\rm Coll}_C(\E)$ reflects
isomorphisms, that is, if a morphism of $C$\nobreakdash-coloured operads induces an isomorphism of the underlying collections,
then it is an isomorphism.

If we forget the symmetric group actions in all the definitions given so far, we obtain
\emph{non\nobreakdash-symmetric coloured collections} and \emph{non\nobreakdash-symmetric coloured operads}.
There is a forgetful functor from $C$\nobreakdash-coloured operads to non\nobreakdash-symmetric $C$\nobreakdash-coloured
operads, which has a left adjoint $\Sigma$ defined by a coproduct
\begin{equation}
\label{leftadjoint}
(\Sigma P)(c_1,\ldots, c_n;c)=\coprod_{\sigma\in\Sigma_n}P(c_{\sigma^{-1}(1)},\ldots, c_{\sigma^{-1}(n)}; c).
\end{equation}

If $C=\{c\}$, then a $C$\nobreakdash-coloured operad $P$ is just an ordinary operad,
where one writes $P(n)$ instead of $P(c,\ldots,c;c)$ with $n$ inputs. Here we recall that
the (non\nobreakdash-symmetric) \textit{associative operad} $\Ass$ is defined as $\Ass(n)=I$ for $n\ge 0$.
Its symmetric version, which we keep denoting by $\Ass$ if no confusion can arise, is
therefore given by $\Ass(n)=I[\Sigma_n]$ for $n\ge 0$,
where $I[\Sigma_n]$ denotes a coproduct of copies of the unit $I$ indexed by $\Sigma_n$,
on which $\Sigma_n$ acts freely by permutations.
The (symmetric) \textit{commutative operad} $\Com$ is defined as $\Com(n)=I$ for $n\ge 0$.

Algebras over coloured operads are defined as follows (by specializing to a single colour,
one recovers the usual notion of algebras over operads).
Let us denote by~$\E^C$ the product category of copies of $\E$ indexed by the set of colours $C$.
For every object ${\bf X}=(X(c))_{c\in C}$ in $\mathcal{E}^C$, a $C$\nobreakdash-coloured operad ${\rm End}({\bf X})$
in $\E$ is defined by
\[
{\rm End}({\bf X})(c_1,\ldots, c_n; c)=\Hom\nolimits_{\mathcal{E}}(X(c_1)\otimes\cdots\otimes X(c_n), X(c)),
\]
where $X(c_1)\otimes\cdots\otimes X(c_n)$ is meant to be $I$ if $n=0$.
The composition product is ordinary composition and the $\Sigma_n$\nobreakdash-action is defined
by permutation of the factors. The $C$\nobreakdash-coloured operad ${\rm End}({\bf X})$ is called the
\emph{endomorphism coloured operad} of the object ${\bf X}$ of $\E^C$.
Similarly, given a morphism ${\bf f}\colon{\bf X}\longrightarrow {\bf Y}$ in $\E^C$, i.e.,
a $C$\nobreakdash-indexed family of maps
$(f_c\colon X(c)\longrightarrow Y(c))_{c\in C}$ in $\E$, there is a $C$\nobreakdash-coloured operad
${\rm End}(\mathbf{f})$, defined as the pullback of the following diagram of $C$\nobreakdash-coloured collections:
\begin{equation}
\xymatrix{
{\rm End}({\bf f})\ar[r] \ar[d] & {\rm End}({\bf X}) \ar[d] \\
{\rm End}({\bf Y})\ar[r] & {\rm Hom}({\bf X}, {\bf Y}),
}
\label{endocolor}
\end{equation}
where the $C$\nobreakdash-coloured collection ${\rm Hom}({\bf X}, {\bf Y})$ is defined as
\[
{\rm Hom}({\bf X}, {\bf Y})(c_1,\ldots, c_n; c)=\Hom\nolimits_{\mathcal{E}}(X(c_1)\otimes\cdots\otimes X(c_n), Y(c)),
\]
and the arrows ${\rm End}({\bf X})\longrightarrow {\rm Hom}({\bf X}, {\bf Y})$ and
${\rm End}({\bf Y})\longrightarrow {\rm Hom}({\bf X}, {\bf Y})$
are induced by ${\bf f}$ by composing on each side.
The $C$\nobreakdash-coloured collection ${\rm End}({\bf f})$ inherits indeed a $C$\nobreakdash-coloured operad structure
from the $C$\nobreakdash-coloured operads
${\rm End}({\bf X})$ and ${\rm End}({\bf Y})$, as observed in \cite[Theorem~3.5]{BM03}.

Given a $C$\nobreakdash-coloured operad $P$ in $\E$, an \emph{algebra over $P$} or a \emph{$P$\nobreakdash-algebra} is an object
${\bf X}=(X(c))_{c\in C}$ of $\mathcal{E}^C$ together with a morphism
\[
P\longrightarrow {\rm End}({\bf X})
\]
of $C$\nobreakdash-coloured operads.
Equivalently, an algebra over a $C$\nobreakdash-coloured operad $P$ can be defined as a family of objects
$X(c)$ in $\E$, for all $c\in C$, together with maps
\[
P(c_1,\ldots, c_n;c)\otimes X(c_1)\otimes\cdots\otimes X(c_n)\longrightarrow X(c)
\]
for every $(n+1)$\nobreakdash-tuple $(c_1,\dots,c_n;c)$, compatible with the symmetric group action,
associativity, and the unit of $P$.

\begin{rem}
\label{nonsymmetric}
If $P$ is a non\nobreakdash-symmetric $C$\nobreakdash-coloured operad $P$, then $P$\nobreakdash-algebras are defined in the same way,
by forgetting the symmetric group action on ${\rm End}({\bf X})$.
If $\Sigma$ denotes the left adjoint (\ref{leftadjoint}) of the forgetful functor
from symmetric to non\nobreakdash-symmetric $C$\nobreakdash-coloured operads, then, for a non\nobreakdash-symmetric $C$\nobreakdash-coloured operad $P$,
there is a bijective correspondence
between the $P$\nobreakdash-algebra structures and the $\Sigma P$\nobreakdash-algebra structures on an object ${\bf X}$ of $\E^C$.
\end{rem}

If ${\bf X}=(X(c))_{c\in C}$ and ${\bf Y}=(Y(c))_{c\in C}$ are $P$\nobreakdash-algebras, a \emph{map of $P$\nobreakdash-algebras}
${\bf f}\colon {\bf X}\longrightarrow {\bf Y}$ is a family of maps $f_c\colon X(c)\longrightarrow Y(c)$
in $\E$, for all $c\in C$, that are compatible with the $P$\nobreakdash-algebra structures on ${\bf X}$ and ${\bf Y}$, i.e.,
the following diagram commutes for all $(n+1)$\nobreakdash-tuples $(c_1,\ldots, c_n; c)$:
\[
\xymatrix{
P(c_1,\ldots,c_n;c)\otimes X(c_1)\otimes\cdots\otimes X(c_n) \ar[r] \ar[d]_{{\rm id}\otimes
f_{c_1}\otimes\cdots\otimes f_{c_n}} &
X(c)\ar[d]^{f_c}\\
P(c_1,\ldots,c_n;c)\otimes Y(c_1)\otimes\cdots\otimes Y(c_n) \ar[r] & Y(c).
}
\]
For a map ${\bf f}\colon{\bf X}\longrightarrow {\bf Y}$ in $\E^C$,
giving a morphism of $C$\nobreakdash-coloured operads
\[
P\longrightarrow {\rm End}(\mathbf{f})
\]
is equivalent by (\ref{endocolor}) to giving a $P$\nobreakdash-algebra structure on ${\bf X}$
and a $P$\nobreakdash-algebra structure on ${\bf Y}$ such that ${\bf f}$ is a map of $P$\nobreakdash-algebras.
The category of $P$\nobreakdash-algebras in $\E$ will be denoted by ${\rm Alg}_P(\E)$.

Given a $C$\nobreakdash-coloured operad $P$ and an object ${\bf X}=(X(c))_{c\in C}$
in $\E^C$, we define the \emph{restricted endomorphism operad} ${\rm End}_{P}({\bf X})$ as follows:
\begin{equation}
\label{restriction}
{\rm End}_{P}({\bf X})(c_1,\dots, c_n;c)=\left\{
\begin{array}{l}
0 \mbox{ if $P(c_1,\ldots, c_n;c)=0$,} \\[0.2cm]
{\rm End}({\bf X})(c_1,\ldots, c_n;c) \mbox{ otherwise.}
\end{array}
\right.
\end{equation}
Thus, there is a canonical inclusion of $C$\nobreakdash-coloured operads
\[
{\rm End}_{P}({\bf X})\longrightarrow {\rm End}({\bf X}),
\]
for which the following holds:

\begin{prop}
If $P$ is a $C$\nobreakdash-coloured operad in $\E$ and ${\bf X}=(X(c))_{c\in C}$ is an object of $\mathcal{E}^C$,
then every morphism $P\longrightarrow {\rm End}({\bf X})$ of $C$\nobreakdash-coloured operads
factors uniquely through ${\rm End}_{P}({\bf X})$.
$\hfill\qed$
\end{prop}

Hence, a $P$\nobreakdash-algebra structure on ${\bf X}$ is precisely given by a morphism of $C$\nobreakdash-coloured
operads $P\longrightarrow {\rm End}_{P}({\bf X})$.
The same holds for non\nobreakdash-symmetric coloured operads,
by replacing endomorphism coloured operads by their non\nobreakdash-symmetric version.

Similarly, if $\bf X$ and $\bf Y$ are objects of ${\E}^C$ and $P$ is a
$C$\nobreakdash-coloured operad, we denote by ${\rm Hom}_{P}({\bf X}, {\bf Y})$
the $C$\nobreakdash-coloured collection defined as
\[
{\rm Hom}_{P}({\bf X},{\bf Y})(c_1,\dots, c_n;c)=\left\{
\begin{array}{l}
0 \mbox{ if $P(c_1,\ldots,c_n;c)=0$,} \\[0.2cm]
{\rm Hom}({\bf X},{\bf Y})(c_1,\ldots, c_n;c) \mbox{ otherwise,}
\end{array}
\right.
\]
and, for a morphism ${\bf f}\colon{\bf X}\longrightarrow {\bf Y}$, we denote by
${\rm End}_{P}({\bf f})$ the pullback of the restricted
endomorphism operads of ${\bf X}$ and ${\bf Y}$ over
${\rm Hom}_{P}({\bf X}, {\bf Y})$, as in (\ref{endocolor}).

\section{Ideals and restriction of colours}
\label{examples}
The following concepts will be useful in our discussion of localization of modules over monoids
and their maps. If $\bf X$ is an algebra over a $C$\nobreakdash-coloured operad $P$, we will need to carry out
certain constructions on some components $X(c)$, but not on others.
For this reason, we give a name to special subsets of the set of colours (depending on~$P$).
Examples will be given later in this section.

\begin{defn}
If $P$ is a $C$\nobreakdash-coloured operad, a subset $J\subseteq C$ is called an \emph{ideal} relative to $P$ if
$P(c_1,\ldots,c_n;c)=0$ whenever $n\ge 1$, $c\in J$, and $c_i\not\in J$ for some $i\in\{1,\ldots,n\}$.
\label{ideal}
\end{defn}

If $\alpha\colon C\longrightarrow D$ is a function between sets of colours, then the
following pair of adjoint functors was discussed in~\cite[\S 1.6]{BM07}:
\begin{equation}
\label{coloradjoint1}
\alpha_{!} : {\rm Oper}_C(\E)\rightleftarrows {\rm Oper}_D(\E) : \alpha^*,
\end{equation}
where the restriction functor $\alpha^*$ is defined as
\[
(\alpha^* P)(c_1,\ldots,c_n;c)=P(\alpha(c_1),\ldots, \alpha(c_n); \alpha(c)).
\]
If $\alpha$ is injective (which is indeed the case in all our applications),
then the left adjoint $\alpha_{!}$ can be made explicit as follows:
\begin{equation}
\label{DCleftadjoint}
(\alpha_! Q)(d_1,\ldots,d_n; d)=
\left\{
\begin{array}{l}
\mbox{$Q(c_1,\ldots, c_n; c)$ if $\alpha(c_i)=d_i$ for all $i$ and $\alpha(c)=d$,} \\[0.1cm]
\mbox{$I$ if $n=1$, $d_1=d$, and $d\not\in \alpha(C)$,} \\[0.1cm]
\mbox{$0$ otherwise.}
\end{array}
\right.
\end{equation}

A function $\alpha\colon C\longrightarrow D$ also defines an adjoint pair in the
corresponding categories of algebras:
\begin{equation}
\label{coloradjoint2}
\alpha_{!} : {\rm Alg}_{\alpha^*P}(\E)\rightleftarrows {\rm Alg}_P(\E) : \alpha^*
\end{equation}
for every $P\in {\rm Oper}_D(\E)$, where $\alpha^*$ is defined as follows. If ${\bf X}$ is
a $P$\nobreakdash-algebra given by a structure morphism $\gamma\colon P\longrightarrow \End({\bf X})$, then
\[
(\alpha^*{\bf X})(c)=X(\alpha(c))
\]
for all $c\in C$, with a structure morphism defined by means of (\ref{coloradjoint1}),
\begin{equation}
\label{alphastar}
\alpha^*\gamma\colon \alpha^*P\longrightarrow \alpha^*\End({\bf X})=\End(\alpha^*{\bf X}).
\end{equation}

The following examples are illustrative.

\subsection{Modules over operad algebras}
\label{mulmod}
Let $P$ be a (one\nobreakdash-coloured) operad in $\mathcal{E}$ and let $\Mod_P$ be a coloured operad with two colours
$C=\{r,m\}$, for which the only nonzero terms are
\[
\Mod_P(r,\stackrel{(n)}{\ldots}, r;r)=P(n)
\]
for $n\ge 0$ and
\[
\Mod_P(c_1,\ldots, c_n;m)=P(n)
\]
for $n\ge 1$ when exactly one $c_i$ is $m$ and the rest (if any) are equal to $r$. Then an algebra over $\Mod_P$
is a pair $(R, M)$ of objects of $\E$
where $R$ is a $P$\nobreakdash-algebra and $M$ is a module over~$R$, i.e., an object equipped with a family of maps
\[
P(n) \otimes R \otimes {\stackrel{(k-1)}{\cdots}} \otimes R \otimes M \otimes R \otimes {\stackrel{(n-k)}{\cdots}}
\otimes R \longrightarrow M
\]
for $n\ge 1$ and $1\le k\le n$, equivariant and compatible with associativity and with the unit of $P$.

If $P=\Ass$, then an algebra over $\Mod_P$ is a pair $(R,M)$ where $R$ is a monoid in~$\E$
and $M$ is an $R$\nobreakdash-bimodule, that is, an object equipped with a right $R$\nobreakdash-action and a left $R$\nobreakdash-action
that commute with each other.
If $P=\Com$, then the corresponding object $R$ is a commutative monoid in $\E$ and $M$ is a module
over it (indistinctly left or right).

The ideals relative to $\Mod_P$ are $C$, $\{r\}$, and $\emptyset$ for all $P$. Note also that, if $\alpha$ denotes
the inclusion of $\{r\}$ into $\{r,m\}$, then
\[
\alpha^*\Mod_P = P
\]
for each operad $P$, and $\alpha^*(R,M)=R$ on the corresponding algebras.

As in \cite{BM07}, we note that there are non\nobreakdash-symmetric coloured operads yielding the notions
of left module and right module. For a (non\nobreakdash-symmetric, one\nobreakdash-coloured) operad $P$,
let $\LMod_P$ be the non\nobreakdash-symmetric $C$\nobreakdash-coloured operad with \hbox{$C=\{r,m\}$} defined by
\[
\LMod_P(r,\stackrel{(n)}{\ldots}, r;r)=P(n), \qquad \LMod_P(r,\stackrel{(n)}{\ldots}, r,m;m)=P(n+1)
\]
for $n\ge 0$, and zero otherwise. Similarly, consider a non\nobreakdash-symmetric coloured operad $\RMod_P$
with two colours $\{s,m\}$ defined by
\[
\RMod_P(s,\stackrel{(n)}{\ldots}, s;s)=P(n), \qquad \RMod_P(m,s,\stackrel{(n)}{\ldots}, s;m)=P(n+1),
\]
for $n\ge 0$, and zero otherwise. If $P=\Ass$ (as a non\nobreakdash-symmetric operad), then
the algebras over $\LMod_P$ are pairs $(R, M)$ of objects of $\E$ where $R$
is a monoid and $M$ supports a left action of $R$, and similarly for $\RMod_P$.

In order to handle $R$\nobreakdash-$S$\nobreakdash-bimodules,
we consider a non\nobreakdash-symmetric coloured operad $\BMod_P$ with three colours $\{r,s,m\}$ and such that
\begin{gather}\notag
\BMod_P(r,\stackrel{(n)}{\ldots}, r;r)=P(n), \qquad \BMod_P(s,\stackrel{(n)}{\ldots}, s;s)=P(n), \\ \notag
\BMod_P(r,\stackrel{(n_1)}{\ldots}, r,m,s,\stackrel{(n_2)}{\ldots}, s;m)=P(n_1+n_2+1),
\end{gather}
if $n,n_1,n_2\ge 0$, and zero otherwise.
The ideals relative to $\BMod_P$ are $C$, $\{r,s\}$, $\{r\}$, $\{s\}$, and $\emptyset$.
Those relative to $\LMod_P$ are $C$, $\{r\}$, and $\emptyset$, and similarly for~$\RMod_P$.

\subsection{Maps of algebras over coloured operads}
\label{mulmoralg}
Let $C$ be any set and $P$ a $C$\nobreakdash-coloured operad.
Let $D=\{0,1\}\times C$ and define a $D$\nobreakdash-coloured operad $\Mor_P$ by
\[
\Mor_P((i_1,c_1),\ldots,(i_n,c_n);(i,c))=\left\{
\begin{array}{l} \mbox{$0$ if $i=0$ and $i_k=1$ for some $k$,} \\[0.2cm]
\mbox{$P(c_1,\ldots, c_n;c)$ otherwise.}
\end{array}
\right.
\]
If ${\bf X}$ is an algebra over $\Mor_P$, then both ${\bf X}_0=(X(0,c))_{c\in C}$ and
${\bf X}_1=(X(1,c))_{c\in C}$ acquire a $P$\nobreakdash-algebra structure by restriction of colours,
since, if $\alpha_i\colon C\longrightarrow D$ denotes the inclusion $\alpha_i(c)=(i,c)$ for $i=0$
and $i=1$, then
\begin{equation}
\label{rhois}
(\alpha_i)^*\Mor_P= P \quad \mbox{and} \quad (\alpha_i)^*{\bf X}={\bf X}_i.
\end{equation}
Furthermore, the $\Mor_P$\nobreakdash-algebra structure on $\bf X$ gives rise to a
map of $P$\nobreakdash-algebras ${\bf f}\colon {\bf X}_0\longrightarrow {\bf X}_1$ as follows.
For each $c\in C$, there is a map $f_c\colon X(0,c)\longrightarrow X(1,c)$ defined as the composite
\begin{equation}
\label{themap}
X(0,c)\longrightarrow \Mor_P((0,c);(1,c))\otimes X(0,c)\longrightarrow X(1,c),
\end{equation}
where the first map is obtained by tensoring the unit $u_c\colon I\longrightarrow P(c;c)$ with $X(0,c)$.

Conversely, given two $P$\nobreakdash-algebras ${\bf X}_0$, ${\bf X}_1$ and
a map of $P$\nobreakdash-algebras ${\bf f}\colon{\bf X}_0\longrightarrow{\bf X}_1$, there is a unique
$\Mor_P$\nobreakdash-algebra structure on ${\bf X}=\left(X_0(c),X_1(c)\right)_{c\in C}$ extending the given $P$\nobreakdash-algebra structures
and for which the distinguished map defined by (\ref{themap}) is the given map $\bf f$.

For example, an algebra ${\bf X}$ over $\Mor_{\Ass}$
is determined by two monoids $X(0)$ and $X(1)$ together with a morphism of monoids $f\colon X(0)\longrightarrow X(1)$.
If $P$ is any one\nobreakdash-coloured operad, then we can write $D=\{0,1\}$, hence recovering \cite[1.5.3]{BM07};
cf.\ also \cite[\S 2]{Mar04}. In this case, the ideals relative to $\Mor_{P}$ are $D$, $\{0\}$, and $\emptyset$.

If $Q=\Mod_P$, as in Subsection~\ref{mulmod}, where $P$ is one\nobreakdash-coloured, then $\Mor_Q$ is a $D$\nobreakdash-coloured operad with $D=\{0,1\}\times\{r,m\}$.
An algebra ${\bf X}$ over $\Mor_Q$ is uniquely determined by two $P$\nobreakdash-algebras $A=X(0,r)$ and $B=X(1,r)$, an $A$\nobreakdash-module \hbox{$M=X(0,m)$},
a $B$\nobreakdash-module $N=X(1,m)$, a map of $P$\nobreakdash-algebras $A\longrightarrow B$, and a map of $A$\nobreakdash-modules $M\longrightarrow N$,
where the $A$\nobreakdash-module structure on $N$ is defined by means of the map $A\longrightarrow B$.
The ideals relative to $\Mor_Q$ are the following subsets:
$D$, $\{(0,r),(0,m),(1,r)\}$, $\{(0,r),(0,m)\}$, $\{(0,r),(1,r)\}$, $\{(0,r)\}$, and $\emptyset$.

\section{Preservation of algebras under localizations}
\label{preservation}
In this section, we study the effect of localizations on structures defined as algebras over coloured operads
in closed symmetric monoidal categories, and describe several special cases.
First, we recall some generalities about localizations.

\label{secloc}
Let $\C$ be any category. A \emph{coaugmented functor}
on $\C$ is a functor $L\colon\C\longrightarrow\C$ together with a natural transformation
$\eta\colon {\rm Id}_{\C}\longrightarrow L$. (This is called a \emph{pointed endofunctor} in other contexts.)
A coaugmented functor $(L,\eta)$ is \emph{idempotent} if $\eta_{LX}=L\eta_X$
and $L\eta_X\colon LX\longrightarrow LLX$ is an isomorphism for every object $X$ in $\C$. Idempotent
coaugmented functors are called \emph{localizations}.

If $(L,\eta)$ is a localization, then the objects isomorphic to $LX$ for some $X$ are called \emph{$L$\nobreakdash-local objects}
and the morphisms $f\colon X\longrightarrow Y$ such that $Lf\colon LX\longrightarrow LY$ is an isomorphism are
called \emph{$L$\nobreakdash-equivalences}. Localizations are characterized by each of two universal properties:
\begin{itemize}
\item[(i)] $\eta_X\colon X\longrightarrow LX$ is initial among morphisms from $X$ to $L$\nobreakdash-local objects;
\item[(ii)] $\eta_X\colon X\longrightarrow LX$ is terminal among $L$\nobreakdash-equivalences with domain $X$.
\end{itemize}
These universal properties ensure that if $f\colon X\longrightarrow Y$ is an $L$\nobreakdash-equivalence and $Y$ is
$L$\nobreakdash-local, then $Y\cong LX$. In fact, the classes of $L$\nobreakdash-local objects and $L$\nobreakdash-equivalences determine each other
by an orthogonality relation. A morphism $f\colon X\longrightarrow Y$ and an object $Z$ in $\C$ are called \emph{orthogonal}
if the induced map
\begin{equation}
\label{orthogonal}
\C(f,Z)\colon\C(Y,Z)\longrightarrow\C(X,Z)
\end{equation}
is a bijection. Using this terminology, a map is an $L$\nobreakdash-equivalence if and only if it is orthogonal to all $L$\nobreakdash-local objects,
and an object is $L$\nobreakdash-local if and only if it is orthogonal to all $L$\nobreakdash-equivalences.
Examples of localization functors on the homotopy
category of spaces or spectra are localization at primes, homological localizations, and, more generally,
$f$\nobreakdash-localizations in the sense of \cite{Far96}.

Here it is convenient to introduce the following convention about extending coaugmented functors from a
category $\E$ to the product category $\E^C$, where $C$ is a set of colours. In some of our
results, it will be necessary to localize a subset of components of an algebra over a $C$\nobreakdash-coloured operad,
but not the rest (for example, we may want to localize an $R$\nobreakdash-module, but not the monoid $R$).
Thus, for a coaugmented functor $(L,\eta)$ on~$\E$, we define partial extensions over $\E^C$ as follows:

\begin{defn}
\label{extensions}
The \textit{extension of $(L,\eta)$ over $\E^C$ away from a subset $J\subseteq C$}
is the coaugmented functor on $\E^C$ ---which we keep denoting by $(L,\eta)$ if no confusion can arise---
given by $L{\bf X}=(L_c X(c))_{c\in C}$ where $L_c$ is the identity functor if $c\in J$ and $L_c=L$ if $c\not\in J$.
Correspondingly, $\eta_{{\bf X}}\colon {\bf X}\longrightarrow L{\bf X}$ is defined by declaring that
$(\eta_{{\bf X}})_c$ is the identity map if $c\in J$ and $(\eta_{{\bf X}})_c=\eta_{X(c)}$ if $c\not\in J$.
\end{defn}

\begin{lem}
Let $\E$ be a closed symmetric monoidal category, $P$ a $C$\nobreakdash-coloured operad in $\E$,
and ${\bf X}$ a $P$\nobreakdash-algebra. Let $(L,\eta)$ be any extension over $\E^C$ of a coaugmented functor on $\E$.
Suppose that the morphism of $C$\nobreakdash-coloured collections
\[
{\rm End}_{P}(L{\bf X})\longrightarrow {\rm Hom}_{P}({\bf X}, L{\bf X})
\]
induced by $\eta_{{\bf X}}$ is an isomorphism.
Then $L{\bf X}$ has a unique $P$\nobreakdash-algebra structure such that $\eta_{{\bf X}}$ is a map of $P$\nobreakdash-algebras.
\label{lem01}
\end{lem}

\begin{proof}
The assumption made implies that the pullback morphism
\[
{\rm End}_{P}(\eta_{\bf X})\longrightarrow {\rm End}_{P}({\bf X})
\]
is an isomorphism of $C$\nobreakdash-coloured collections, and therefore it is an isomorphism of $C$\nobreakdash-coloured operads.
Composing the inverse of this isomorphism
with the morphism $\gamma\colon P\longrightarrow {\rm End}_{P}({\bf X})$ that endows ${\bf X}$
with its $P$\nobreakdash-algebra structure yields a morphism
$P\longrightarrow {\rm End}_{P}(L{\bf X})$ as depicted in the diagram
\[
\xymatrix{
 & {\rm End}_{P}(\eta_{\bf X}) \ar[r]\ar[d] & {\rm End}_{P}(L{\bf X})\ar[d] \\
 P\ar[r]^-{\gamma} & {\rm End}_{P}(\bf X) \ar@/^/@{.>}[u] \ar[r] & {\rm Hom}_{P}({\bf X}, L{\bf X}).
}
\]
In this way, $L{\bf X}$ acquires a $P$\nobreakdash-algebra structure. The fact that this $P$\nobreakdash-algebra structure
morphism factors through ${\rm End}_{P}(\eta_{\bf X})$ implies precisely
that $\eta_{\bf X}$ is a map of $P$\nobreakdash-algebras. Furthermore, the $P$\nobreakdash-algebra structure on
$L{\bf X}$ is unique with this property, by the universal property of the pullback.
\end{proof}

Our main source of applications of this result corresponds to the
situation where $(L,\eta)$ is of a special kind. We will ask it to satisfy
an orthogonality condition that is stronger than~(\ref{orthogonal}),
but nonetheless holds in our examples in this section.

\begin{defn}
\label{closed}
We say that a localization $(L,\eta)$ on a closed symmetric monoidal
category $\E$ is \emph{closed} if, for
every $L$\nobreakdash-equivalence $f\colon X\longrightarrow Y$ and
every $L$\nobreakdash-local object~$Z$, the map
\[
\Hom\nolimits_{\E}(f,Z)\colon\Hom\nolimits_{\E}(Y,Z)\longrightarrow \Hom\nolimits_{\E}(X,Z)
\]
is an isomorphism in $\E$.
\end{defn}

For such a functor $L$, if $f_1\colon X_1\longrightarrow Y_1$
and $f_2\colon X_2\longrightarrow Y_2$ are $L$\nobreakdash-equivalences, then the tensor product $f_1\otimes f_2$ is again an
$L$\nobreakdash-equivalence, since, by the hom\nobreakdash-tensor adjunction,
\begin{multline}
\label{shift}
\Hom\nolimits_{\E}(Y_1\otimes Y_2, Z)\cong \Hom\nolimits_{\E}(Y_1, \Hom\nolimits_{\E}(Y_2, Z))
\\ \cong\Hom\nolimits_{\E}(Y_1, \Hom\nolimits_{\E}(X_2, Z))
\cong \Hom\nolimits_{\E}(Y_1\otimes X_2, Z)
\end{multline}
for every $L$\nobreakdash-local object $Z$, and similarly in order to replace $Y_1$ by~$X_1$.

In the rest of this section, we will only consider closed localizations.
The following theorem states that the assumptions of Lemma \ref{lem01} hold for these functors.

\begin{thm}
Let $\mathcal{E}$ be a closed symmetric monoidal category, $P$ a $C$\nobreakdash-coloured operad in $\E$,
and $(L,\eta)$ the extension over $\E^C$ of a closed localization away from an ideal $J\subseteq C$
relative to~$P$.
If $\bf X$ is a $P$\nobreakdash-algebra, then $L{\bf X}$ has a unique $P$\nobreakdash-algebra structure such that
$\eta_{{\bf X}}$ is a map of $P$\nobreakdash-algebras.
\label{thmideal}
\end{thm}
\begin{proof}
By Lemma \ref{lem01}, it is enough to show that the
morphism of $C$\nobreakdash-coloured collections
\[
{\rm End}_{P}(L{\bf X})\longrightarrow {\rm Hom}_{P}({\bf X}, L{\bf X})
\]
induced by $\eta_{{\bf X}}$ is an isomorphism. Since $J$ is an ideal, we need only consider the values of these collections
on tuples $(c_1,\ldots, c_n;c)$ for which $c\in J$ and $c_i\in J$ for all~$i$, or $c\not\in J$.
In the first case, the map
\[
\xymatrix{
\Hom\nolimits_{\mathcal{E}}(L_{c_1}X(c_1)\otimes\cdots\otimes L_{c_n}X(c_n), L_cX(c))\ar[d] \\
\Hom\nolimits_{\mathcal{E}}(X(c_1)\otimes\cdots\otimes X(c_n), L_cX(c))
}
\]
is trivially an isomorphism because $L_c$ and all the $L_{c_i}$ are identity functors, according to Definition \ref{extensions}.
If $c\not\in J$, then $L_c=L$, and
the isomorphism follows from the fact that the tensor product of $L$\nobreakdash-equivalences is an $L$\nobreakdash-equivalence.
\end{proof}

Observe that Lemma \ref{lem01} and Theorem \ref{thmideal} remain true if the coloured operad $P$ is non\nobreakdash-symmetric.
At any moment, if necessary, we may replace $P$
by its symmetric version $\Sigma P$, since both yield the same class of algebras (see Remark \ref{nonsymmetric}).

\begin{cor}
Let $(L,\eta)$ be a closed localization on a closed symmetric monoidal category $\E$.
\begin{itemize}
\item[(i)] If $R$ is a monoid in $\E$, then $LR$ has a unique monoid structure such that $\eta_R\colon R\longrightarrow LR$
is a morphism of monoids. If $R$ is commutative, then $LR$ is also commutative.
\item[(ii)] If $f\colon R_1\longrightarrow R_2$ is a morphism of monoids in $\E$, then $Lf\colon LR_1\longrightarrow LR_2$
is also a morphism of monoids.
\item[(iii)] If $R$ is a monoid in $\E$ and $M$ is a left $R$\nobreakdash-module, then $LM$ has a unique left $R$\nobreakdash-module structure such that
$\eta_M\colon M\longrightarrow LM$ is a morphism of $R$\nobreakdash-modules. Moreover, $LM$ also has a unique left $LR$\nobreakdash-module structure
extending the $R$\nobreakdash-module structure. The same statements are true for right $R$\nobreakdash-modules.
\item[(iv)] If $R$ and $S$ are monoids in $\E$ and $M$ is an $R$\nobreakdash-$S$\nobreakdash-bimodule, then $LM$ has a unique $R$\nobreakdash-$S$\nobreakdash-bimodule structure
such that $\eta_M\colon M\longrightarrow LM$ is a morphism of $R$\nobreakdash-$S$\nobreakdash-bimodules. Moreover, $LM$ also has unique $R$\nobreakdash-$LS$\nobreakdash-bimodule,
$LR$\nobreakdash-$S$\nobreakdash-bimodule, and $LR$\nobreakdash-$LS$\nobreakdash-bimodule structures that extend the given \linebreak $R$\nobreakdash-$S$\nobreakdash-bimodule structure.
\item[(v)] If $f\colon M_1\longrightarrow M_2$ is a morphism of left $R$\nobreakdash-modules, where $R$ is a monoid in $\E$,
then $Lf\colon LM_1\longrightarrow LM_2$ is a morphism of left $R$\nobreakdash-modules and a morphism of left $LR$\nobreakdash-modules.
The analogous statements are true for right $R$\nobreakdash-modules and for $R$\nobreakdash-$S$\nobreakdash-bimodules, where $S$ is another monoid.
\end{itemize}
\label{cormodules}
\end{cor}

\begin{proof}
This follows from Theorem \ref{thmideal} using the coloured operads of Subsections \ref{mulmod} and
\ref{mulmoralg}, by choosing a suitable ideal in each case.
In part~(i), pick the operads $\Ass$ and $\Com$, viewed as coloured operads with one colour, together with the ideal $J=\emptyset$
in each case. In part~(ii), pick the coloured operad $\Mor_{\Ass}$ of Subsection~\ref{mulmoralg}, together with
the ideal $J=\emptyset$. (Note that $f\longrightarrow Lf$ is a commutative diagram of morphisms of monoids,
and therefore $LR_1$ and $LR_2$ are equipped with the monoid structure given by~(i).)
In part~(iii), use first the coloured operad $\LMod_{\Ass}$ of Subsection \ref{mulmod} with
the ideal $J=\{r\}$ in order to endow $LM$ with a left $R$\nobreakdash-module structure, and choose the ideal
$J=\emptyset$ to endow $LM$ with an $LR$\nobreakdash-module structure extending the previous $R$\nobreakdash-module structure.
Similarly for right modules. For bimodules, in part~(iv), use the coloured operad $\BMod_{\Ass}$ with each of the ideals
$J=\{r,s\}$, $J=\{r\}$, $J=\{s\}$, and $J=\emptyset$, in order to endow $LM$ with an $R$\nobreakdash-$S$\nobreakdash-bimodule structure,
an $R$\nobreakdash-$LS$\nobreakdash-bimodule structure, an $LR$\nobreakdash-$S$\nobreakdash-bimodule structure, and an $LR$\nobreakdash-$LS$\nobreakdash-bimodule structure, respectively.
In part (v), use the coloured operad $\Mor_Q$ described in Subsection~\ref{mulmoralg} for $Q=\LMod_{\Ass}$, with the ideal $J=\{(0,r),(1,r)\}$
in order to infer that $Lf$ is a morphism of $R$\nobreakdash-modules, and $J=\emptyset$
in order to infer that $Lf$ is a morphism of $LR$\nobreakdash-modules. Similarly with $Q=\RMod_{\Ass}$ for right $R$\nobreakdash-modules
and $Q=\BMod_{\Ass}$ for $R$\nobreakdash-$S$\nobreakdash-bimodules. As in~(ii), the module or bimodule structures
on $LM_1$ and $LM_2$ are those given by (iii) or~(iv), since $f\longrightarrow Lf$ is a commutative diagram of morphisms.
\end{proof}

\begin{cor}
Let $(L,\eta)$ be a closed localization on a closed symmetric monoidal category $\E$.
If $P$ is a $C$\nobreakdash-coloured operad in~$\E$, then the $C$\nobreakdash-coloured collection $LP$ defined as
\[
(LP)(c_1,\ldots, c_n; c)=L(P(c_1,\ldots, c_n;c))
\]
has a unique $C$\nobreakdash-coloured operad structure such that the map $P\longrightarrow LP$
induced by~$\eta$ is a morphism of $C$\nobreakdash-coloured operads.
\end{cor}

\begin{proof}
For each set $C$, there is a coloured operad
whose algebras are precisely the $C$\nobreakdash-coloured operads in $\mathcal{E}$; for a description,
see \cite[Examples 1.5.6 and 1.5.7]{BM07}. The statement follows by applying Theorem~\ref{thmideal}
to this coloured operad, with the empty ideal.
\end{proof}

As we next explain, Corollary~\ref{cormodules} implies a number of known results about
preservation of certain structures under localizations. Subsections \ref{spaces} and \ref{exactlocs}
refer to the homotopy categories of spaces
and spectra, respectively, in which the corresponding results are weak forms of the stronger results
described in Section~\ref{principal}.

\subsection{Discrete rings and modules}
\label{discrete}
In the category $\Ab$ of abelian groups, given a homomorphism $f\colon A\longrightarrow B$,
an abelian group $G$ is called \emph{$f$\nobreakdash-local} if it is orthogonal to $f$, that is, if
\[
\Ab(f,G)\colon \Ab(B,G)\longrightarrow\Ab(A,G)
\]
is a bijection. A homomorphism is called an \emph{$f$\nobreakdash-equivalence} if it is
orthogonal to all $f$\nobreakdash-local groups. 
By general results about locally presentable categories
(see \cite[Theorem~1.39]{AR94}), there is a localization functor $(L_{f},\eta)$
for every $f$ on the category of abelian groups,
called \emph{$f$\nobreakdash-localization}, such that $\eta_G\colon G\longrightarrow L_{f}G$
is an $f$\nobreakdash-equivalence into an $f$\nobreakdash-local group for all $G$.

Thus $L_f$ is a closed localization if we endow the category of abelian groups with the
closed symmetric monoidal structure given by the tensor product over~$\Z$ and the canonical
enrichment of $\Ab$ over itself.
Hence, we infer from Corollary~\ref{cormodules} the following observation made in \cite[Theorems 3.8 and 3.9]{Cas00},
where by a ring we mean an associative ring $R$ with a unit morphism $\Z\longrightarrow R$
(so the zero ring is not excluded):

\begin{prop}
In the category of abelian groups, every $f$\nobreakdash-localization preserves the classes of rings,
commutative rings, left or right modules over a ring, and bimodules over rings. \qed
\end{prop}

\subsection{$H$\nobreakdash-spaces}
\label{spaces}
Let ${\rm Ho}$ be the homotopy category of $k$\nobreakdash-spaces with the
Quillen model structure (as in \cite[II.3]{Qui67}), or the homotopy category of simplicial sets
with the Kan model structure.
Each of these (equivalent) categories is closed symmetric monoidal
with the corresponding derived product as tensor product, the one\nobreakdash-point space as unit,
and the derived mapping space $\map(-,-)$ as internal hom; cf.~\cite[Theorem 4.3.2]{Hov99}.
A monoid in ${\rm Ho}$ is a homotopy associative $H$\nobreakdash-space, i.e.,
a space $X$ together with a
multiplication map $X\times X\longrightarrow X$ that is associative up to homotopy
and with a homotopy unit.

For every map $f$ between spaces there is an $f$\nobreakdash-localization functor on ${\rm Ho}$
(see \cite{Far96} or \cite[4.1.1]{Hir03}), which is closed by construction.
Hence, the following fact is deduced from Corollary~\ref{cormodules}:

\begin{prop}
Every $f$\nobreakdash-localization on spaces preserves the classes
of homotopy associative $H$\nobreakdash-spaces and homotopy commutative $H$\nobreakdash-spaces. $\hfill \qed$
\end{prop}

\subsection{Homotopy ring spectra and homotopy module spectra}
\label{exactlocs}
Let ${\rm Ho}^{\rm s}$ be the stable homotopy category of Adams\nobreakdash--Boardman, which is
closed symmetric monoidal with the derived smash product as tensor product,
the sphere spectrum as unit, and the derived function spectrum $F(-,-)$ as internal hom.

A monoid in ${\rm Ho}^{\rm s}$ is a homotopy ring spectrum and a module over a monoid is a homotopy module spectrum.
Hence, the following result, which extends \cite[Theorem 4.2]{CG05}, is a consequence of Corollary \ref{cormodules}:

\begin{prop}
If $(L, \eta)$ is a closed localization on spectra and $R$ is a homotopy ring spectrum,
then $LR$ admits a unique homotopy ring structure such that
$\eta_R\colon R\longrightarrow LR$ is a homotopy ring map. If $M$ is a homotopy $R$\nobreakdash-module,
then $LM$ admits a unique homotopy $R$\nobreakdash-module structure such that
\hbox{$\eta_M\colon M\longrightarrow LM$} is a homotopy $R$\nobreakdash-module map,
and $LM$ admits a unique homotopy $LR$\nobreakdash-module structure
extending the $R$\nobreakdash-module structure. $\hfill \qed$
\label{LRLM}
\end{prop}

Note that, in particular, if $M$ is a homotopy $R$\nobreakdash-module and $LR\simeq 0$, then 
we deduce that $LM\simeq 0$ as well.

Examples of closed localizations on ${\rm Ho}^{\rm s}$ are stable
$f$\nobreakdash-localizations when $f$ is a wedge of maps $\{\Sigma^k g\}$ for all $k\in\Z$
and some map~$g$; see \cite[Theorem~2.7]{CG05}.
Homological localizations are of this kind.

A localization on ${\rm Ho}^{\rm s}$ is closed if and only if
it commutes with suspension, and this is equivalent to the property of
preserving cofibre sequences
(see also the remarks made in Subsection~\ref{nowspectra} below).
Thus it is important to distinguish closed localizations 
from other localizations on ${\rm Ho}^{\rm s}$ 
that do not preserve cofibre sequences, such as Postnikov sections.
Indeed, Proposition~\ref{LRLM} does not hold if $L$ is a Postnikov section
and $R=K(n)$; see (\ref{PKn}) below for details.

\section{Homotopical localization functors}
\label{holocs}
When one works with model categories, orthogonality between maps and objects is more conveniently
discussed in terms of
\emph{homotopy function complexes}. This is a stronger notion than orthogonality defined in terms
of homotopy classes of maps,
and distinct from orthogonality defined in terms of an internal hom (if available), in general.
A \emph{homotopy function complex} in a model category $\M$ is a
functorial choice, for every two objects $X$ and $Y$ in $\M$, of a fibrant simplicial set $\map(X,Y)$
whose homotopy type is the same
as the diagonal of the bisimplicial set $\M(X^*, Y_*)$ where $X^*\longrightarrow X$ is a cosimplicial resolution
of $X$ and $Y\longrightarrow Y_*$ is a simplicial resolution of $Y$, as defined, e.g., in \cite[16.1]{Hir03}.
Thus, the homotopy type of $\map(X,Y)$ does not change if we replace $X$ or $Y$ by weakly equivalent objects, and
$\pi_0\map(X,Y)$ is in natural bijective correspondence with the set $[X,Y]$ of morphisms from $X$ to $Y$
in $\Ho(\M)$. For more details, see \cite[Theorem 17.7.2]{Hir03}.
The existence of homotopy function complexes
in every model category is proved in \cite[5.4]{Hov99} and \cite[17.3]{Hir03}.

Recall that a \emph{simplicial category}
is a category $\C$ equipped with an enrichment, a tensor and a cotensor over the category of simplicial sets.
Thus, there are functors
\begin{gather}\notag
\Map(-,X)\colon \C^{\rm op}\longrightarrow \sSets; \quad
-\boxtimes X\colon \sSets\longrightarrow \C; \quad
X^{(-)}\colon \sSets^{\rm op}\longrightarrow \C, \notag
\end{gather}
for every object $X$ of $\C$, satisfying certain compatibility relations. See \cite{GJ99} or \cite[\S 9.1]{Hir03} for details.
Among these, for every two objects $X$ and $Y$ of $\C$ and every simplicial set $K$, there are natural bijections
\begin{equation}
\C(K\boxtimes X, Y)\cong \sSets(K,\Map(X,Y))\cong \C(X, Y^K).
\label{adj}
\end{equation}

A \emph{simplicial model category} is a model category $\M$ that is also a simplicial category and satisfies Quillen's SM7 axiom:
If $f\colon X\longrightarrow Y$ is a cofibration in $\M$ and $g\colon U \longrightarrow V$ is a fibration in $\M$, then the induced map
\[
\Map(Y,U)\longrightarrow \Map(Y,V)\times_{\Map(X,V)}\Map(X,U)
\]
is a fibration of simplicial sets that is trivial if $f$ or $g$ is trivial.

If $\M$ is a simplicial model category, then $\map(X,Y)=\Map(QX,FY)$ defines
a homotopy function complex, where $\Map(-,-)$ denotes the simplicial enrichment,
$Q$ is a functorial cofibrant replacement and $F$ is a functorial fibrant replacement.

Now let $\M$ be any model category with a choice of homotopy function complexes denoted by $\map(-,-)$.
We will also assume that $\M$ has functorial factorizations, as in \cite{Hov99} and \cite{Hir03}.
A morphism $f\colon X\longrightarrow Y$ and an object $Z$ are called \emph{simplicially orthogonal} if the induced map
\begin{equation}
\label{simport}
\map(f,Z)\colon \map(Y,Z)\longrightarrow \map(X,Z)
\end{equation}
is a weak equivalence of simplicial sets. This form of orthogonality is
used in the following definition.

\begin{defn}
A \emph{homotopical localization} on a model category $\M$ with
homotopy function complexes $\map(-,-)$ is a functor $L\colon\M\longrightarrow\M$ that preserves weak equivalences
and takes fibrant values, together with a natural transformation
\hbox{$\eta\colon {\rm Id}_{\M}\longrightarrow L$}
such that, for every object $X$, the following hold:
\begin{itemize}
\item[(i)] $L\eta_X\colon LX\longrightarrow LLX$ is a weak equivalence;
\item[(ii)] $\eta_{LX}$ and $L\eta_X$ are equal in the homotopy category $\Ho(\M)$;
\item[(iii)] $\eta_X\colon X\longrightarrow LX$ is a cofibration such that the map
\[
\map(\eta_X,LY)\colon\map(LX, LY)\longrightarrow \map(X, LY)
\]
is a weak equivalence of simplicial sets for all $Y$.
\end{itemize}
\end{defn}

The condition that $L$ takes fibrant values and the condition that $\eta_X$ is a cofibration
for all $X$ are technical, yet useful in practice. None of the two imposes a restriction on
the definition, since, if $L$ does not take fibrant values, then we may replace it by~$FL$,
where $F$ is a fibrant replacement functor, and we may also decompose $\eta_X$ functorially 
into a cofibration followed by a trivial fibration for all~$X$,
\[ X\stackrel{\xi_X}{\longrightarrow} KX\stackrel{\nu_X}{\longrightarrow} LX. \]
Then $K$ becomes a functor and
$\xi\colon {\rm Id}_{\M}\longrightarrow K$ a natural transformation for which
(i), (ii) and (iii) hold. (The condition $\xi_{KX}\simeq K\xi_X$ is satisfied
since $\eta_{LX}\simeq L\eta_X$ and $\nu_{LX}\circ K\nu_X=L\nu_X\circ \nu_{KX}$,
as $\nu\colon K\longrightarrow L$ is also a natural transformation.)

Every homotopical localization becomes just an idempotent functor when we pass to the homotopy category ${\rm Ho}(\M)$,
since $\pi_0\map(X,Y)\cong [X,Y]$.
If $(L,\eta)$ is a homotopical localization, then the fibrant objects of $\M$ 
weakly equivalent to $LX$ for some $X$ are called \emph{$L$\nobreakdash-local}, and the maps
$f\colon X\longrightarrow Y$ such that $Lf\colon LX\longrightarrow LY$ is a weak equivalence
are called \emph{$L$\nobreakdash-equivalences}. In addition to orthogonality in ${\rm Ho}({\M})$,
$L$\nobreakdash-local objects and $L$\nobreakdash-equivalences are simplicially orthogonal as defined in (\ref{simport}),
and in fact a fibrant object is $L$\nobreakdash-local if and only if it is simplicially orthogonal to all $L$\nobreakdash-equivalences,
while a map is an $L$\nobreakdash-equivalence if and only if it is simplicially orthogonal to all $L$\nobreakdash-local objects.

If $(L,\eta)$ is a homotopical localization on simplicial sets or
$k$\nobreakdash-spaces, then either all nonempty
$L$\nobreakdash-local spaces are weakly equivalent to a point, or
all $L$\nobreakdash-equivalences are bijective on connected
components. The following argument to prove this claim is well
known. If there is an $L$\nobreakdash-equivalence that is not
bijective on connected components, then it has a retract of the form
$S^0\longrightarrow *$ or $*\longrightarrow S^0$ (besides the
trivial case $\emptyset\longrightarrow *$). Since every retract of
an $L$\nobreakdash-equivalence is an $L$\nobreakdash-equivalence, it
follows that, if $X$ is $L$\nobreakdash-local, then $X$ has the same
homotopy type as $X\times X$, which implies that $X$ is weakly
contractible (or empty). Because of this observation, we will assume
throughout that $\eta$ induces an isomorphism $\pi_0(X)\cong
\pi_0(LX)$ for any homotopical localization $L$ and all spaces~$X$.

Sufficient conditions for a localization on ${\rm Ho}(\M)$ in order that it be induced
by a homotopical localization on $\M$ were given in \cite{CC06}.
Most localizations encountered in practice, including all $f$\nobreakdash-localizations in the sense of \cite{Far96},
are homotopical localizations.
In fact, if $\M$ is a left proper, cofibrantly generated, locally presentable simplicial model category without empty hom\nobreakdash-sets,
and one assumes the validity of Vop\v{e}nka's principle from set theory, then every localization on ${\rm Ho}(\M)$ comes
from an $f$\nobreakdash-localization on $\M$ for some map~$f$; see~\cite[Theorem 2.3]{CC06}.

\section{Model structures on categories of operads}
\label{modelstructures}
Before presenting our main results, we still need to recall from \cite{BM03} and \cite{BM07} the terminology and basic properties of
a model structure for the category of coloured operads over a fixed set of colours.

A \emph{monoidal model category} $\E$ is a closed symmetric monoidal category
with a model structure that satisfies the \emph{pushout\nobreakdash-product axiom} (see \cite[\S 4.2]{Hov99}, \cite{SS00}):
If $f\colon X\longrightarrow Y$ and $g\colon U\longrightarrow V$ are cofibrations in $\E$, then the induced map
\[
(X\otimes V)\coprod_{X\otimes U} (Y\otimes U)\longrightarrow Y\otimes V
\]
is a cofibration that is trivial if $f$ or $g$ is trivial. We will also assume that the unit $I$ of $\E$ is cofibrant.
Using the adjunction between $\otimes$ and $\Hom_{\E}(-,-)$, one obtains the following equivalent formulation of the
pushout\nobreakdash-product axiom: If $f\colon X\longrightarrow Y$ is a cofibration in $\E$ and $g\colon U\longrightarrow V$
is a fibration in $\E$, then the induced map
\[
\Hom\nolimits_{\E}(Y,U)\longrightarrow
\Hom\nolimits_{\E}(Y,V)\times_{\Hom\nolimits_{\E}(X,V)}\Hom\nolimits_{\E}(X,U)
\]
is a fibration in $\E$ that is trivial if either $f$ or $g$ is trivial.

Let $\E$ be a monoidal model category. If $\E$ is cofibrantly generated, then,
as explained in \cite[\S 3]{BM03} and \cite[\S 3]{BM07},
the category of $C$\nobreakdash-coloured collections in~$\E$ admits a model
structure in which a morphism $K\longrightarrow L$ is a weak equivalence (resp.\ a fibration) if and only if for each
tuple of colours $(c_1,\ldots, c_n; c)$ the map
\[
K(c_1,\ldots, c_n; c)\longrightarrow L(c_1,\ldots, c_n; c)
\]
is a weak equivalence (resp.\ a fibration) in $\E$.
This model structure can be transferred along the free\nobreakdash-forgetful adjunction
\begin{equation}
\label{transferred}
F : {\rm Coll}_C(\E)\leftrightarrows {\rm Oper}_C(\E) : U
\end{equation}
to provide a model structure on the category of $C$\nobreakdash-coloured operads, under suitable assumptions on the category
$\E$, including still the assumption that $\E$ be cofibrantly generated; see \cite[Theorem~3.2]{BM03}
and \cite[Theorem~2.1 and Example~1.5.7]{BM07}.

The monoidal model categories of $k$\nobreakdash-spaces (with the Quillen model structure)
and simplicial sets satisfy these assumptions. Thus,
in any of these categories, a morphism of $C$\nobreakdash-coloured operads $f\colon P\longrightarrow Q$ is a weak equivalence
(resp.\ a fibration) if and only if $Uf$ is a weak equivalence (resp.\ a fibration) of $C$\nobreakdash-coloured collections.
Cofibrations are defined by the left lifting property with respect to the trivial fibrations.
A $C$\nobreakdash-coloured operad $P$ is cofibrant if the unique morphism $I_C\longrightarrow P$ is a cofibration, where $I_C$
is the initial $C$\nobreakdash-coloured operad defined by $I_C(c;c)=I$ for all $c$, and zero otherwise.

The $W$\nobreakdash-construction of Boardman\nobreakdash--Vogt for
$C$\nobreakdash-coloured operads (see \cite{BV73}, \cite{Vog03},
\cite{BM06}, \cite[\S 3]{BM07}) provides a cofibrant replacement for
$C$\nobreakdash-coloured operads $P$ whose underlying
$C$\nobreakdash-coloured collection (pointed by the unit) is
$\Sigma$-cofibrant; that is, the unit map $I\longrightarrow P(c;c)$ is a cofibration for all~$c$
and $P(c_1,\ldots,c_n;c)$ is cofibrant as a
$\Sigma_{c_1,\ldots,c_n}$\nobreakdash-space for all
$(c_1,\ldots,c_n;c)$, where $\Sigma_{c_1,\ldots,c_n}$ denotes the
subgroup of $\Sigma_n$ leaving $(c_1,\ldots,c_n)$ invariant; see
\cite[Theorem~3.5]{BM07}. This was implicit in \cite{BV73} for
topological operads, and further developed in~\cite{Vog03} for the
category of $k$\nobreakdash-spaces with the Str{\o}m model
structure.

From now on we will only consider categories of coloured operads admitting the model structure
transferred along (\ref{transferred}).
Under this assumption, for every function $\alpha\colon C\longrightarrow D$, the adjunction
\[
\alpha_{!} : {\rm Oper}_C(\E)\rightleftarrows {\rm Oper}_D(\E) : \alpha^*
\]
given by (\ref{coloradjoint1}) is a Quillen pair, since $\alpha^*$ preserves fibrations and weak equivalences.

Given a $C$\nobreakdash-coloured operad $P$, a \textit{cofibrant resolution} of $P$ is a trivial fibration of
$C$\nobreakdash-coloured operads $P_{\infty}\longrightarrow P$ where $P_{\infty}$ is cofibrant.
(For notational convenience, we also say that $P_{\infty}$ is a cofibrant resolution of $P$.)
Throughout we denote by $A_{\infty}$ an arbitrary but fixed cofibrant resolution of~$\Ass$,
and by $E_{\infty}$ a cofibrant resolution of $\Com$. (It is common practice to denote
by $A_{\infty}$ any non\nobreakdash-symmetric operad that is weakly equivalent to $\Ass$, and by $E_{\infty}$ any
operad that is weakly equivalent to $\Com$; here we assume them cofibrant for
simplicity in the statement of our results.)

We consider two important special cases where change of colours plays a role.

\subsection{$P_{\infty}$\nobreakdash-modules}
\label{ainftymods}
Let $P$ be any (one\nobreakdash-coloured) operad, and let $P_{\infty}\longrightarrow P$
be a cofibrant resolution. As explained in Subsection~\ref{mulmod},
$\Mod_P$ is a $C$\nobreakdash-coloured operad with $C=\{r,m\}$.
Let $\alpha$ denote the inclusion of $\{r\}$ into $C$.
If
\[
(\Mod_P)_{\infty}\longrightarrow \Mod_P
\]
is a cofibrant resolution of $\Mod_P$, then,
since $\alpha^*$ preserves trivial fibrations,
\[
\alpha^*(\Mod_P)_{\infty}\longrightarrow \alpha^*\Mod_P=P
\]
is a trivial fibration.
Hence there is a lifting (unique up to homotopy)
\begin{equation}
\label{jajaja}
\xymatrix{
& \alpha^*(\Mod_P)_{\infty} \ar[d] \\
P_{\infty} \ar@{.>}[ur] \ar[r] & P.
}
\end{equation}

If a pair $(R,M)$ is a $(\Mod_P)_{\infty}$\nobreakdash-algebra, then
$R=\alpha^*(R,M)$ is an algebra over $\alpha^*(\Mod_P)_{\infty}$
by (\ref{alphastar}), and hence a $P_{\infty}$\nobreakdash-algebra via (\ref{jajaja}).
Although the second component $M$ need not be a module,
we call it a \textit{$P_{\infty}$\nobreakdash-module} over~$R$.

\subsection{$P_{\infty}$\nobreakdash-maps}
\label{pinftymaps}
Let $P$ be a $C$\nobreakdash-coloured operad where $C$ is any set of colours, and
choose a cofibrant resolution $\varphi\colon P_{\infty}\longrightarrow P$.
Let $\Mor_{P}$ be as defined in Subsection~\ref{mulmoralg}, with $D=\{0,1\}\times C$.
For $i\in\{0,1\}$, let $\alpha_i\colon C\longrightarrow D$
be the functions defined as $\alpha_0(c)=(0,c)$ and $\alpha_1(c)=(1,c)$. Thus
$(\alpha_i)^*\Mor_P=P$ for both $i=0$ and $i=1$.
If $\Phi\colon \left(\Mor_P\right)_{\infty}\longrightarrow \Mor_P$ is a cofibrant resolution of $\Mor_P$, then,
as in~(\ref{jajaja}), there are morphisms (in fact, weak equivalences) of $C$-coloured operads
\begin{equation}
\label{jajajaja}
\xymatrix{
& (\alpha_i)^*\left(\Mor_P\right)_{\infty} \ar[d]^{(\alpha_i)^*\Phi} \\
P_{\infty} \ar@{.>}[ur]^-{\lambda_i} \ar[r]^-{\varphi} & P
}
\end{equation}
for $i=0$ and $i=1$, unique up to homotopy, rendering the triangle commutative.
Therefore, by (\ref{rhois}),
an algebra $\bf X$ over $\left(\Mor_P\right)_{\infty}$ gives rise to a pair of
$P_{\infty}$\nobreakdash-algebras
$({\bf X}_0,{\bf X}_1)$ with additional structure linking them, which is weaker than
a morphism of $P_{\infty}$\nobreakdash-algebras. Specifically,
since the unit $I$ of $\E$ is cofibrant, we may choose, for each $c\in C$, a lifting
\begin{equation}
\label{weakunits}
\xymatrix{
& \left(\Mor_P\right)_{\infty}((0,c);(1,c)) \ar[d] \\
I \ar[r]_-{u_c} \ar@{.>}[ur] & \Mor_P((0,c);(1,c))
}
\end{equation}
where $u_c$ is the map considered in (\ref{themap}). The lifting is not unique, but it is unique
up to homotopy. Hence, the composites
\[
X(0,c)\longrightarrow \left(\Mor_P\right)_{\infty}((0,c);(1,c))\otimes X(0,c)\longrightarrow X(1,c)
\]
yield together a homotopy class of maps ${\bf X}_0\longrightarrow{\bf X}_1$.
Each of these will be called a \textit{$P_{\infty}$\nobreakdash-map}.
This generalizes the notion of $A_{\infty}$\nobreakdash-map discussed in \cite[I.3]{BV73} and \cite[2.9]{MSS02}.
A lifting similar to (\ref{weakunits}) in the topological case was considered by Schw\"anzl and Vogt
in the context of~\cite{SV88}.
Note that there is also a lifting
\begin{equation}
\label{lifting}
\xymatrix{
& \Mor_{P_{\infty}} \ar[d] \\
\left(\Mor_{P}\right)_{\infty} \ar[r]^-{\Phi} \ar@{.>}[ur]^{\Psi} & \Mor_{P},
}
\end{equation}
since the vertical arrow is a trivial fibration of $C$\nobreakdash-coloured operads
and $\left(\Mor_{P}\right)_{\infty}$ is cofibrant. Hence, every morphism
of $P_{\infty}$\nobreakdash-algebras admits a $P_{\infty}$\nobreakdash-map structure.

For later use, we remark that (\ref{jajaja}) and (\ref{lifting}) yield, for $i=0$ and $i=1$,
\begin{equation}
\label{fancy}
\varphi\circ ((\alpha_i)^*\Psi)\circ \lambda_i = ((\alpha_i)^* \Phi) \circ \lambda_i=\varphi.
\end{equation}

\pagebreak

\section{Preservation of structures in monoidal model categories}
\label{principal}

In Section~\ref{preservation} we saw
that monoids (including homotopy associative $H$\nobreakdash-spaces and homotopy ring spectra), modules over monoids, and morphisms
between these, are preserved by closed localization functors in the unstable homotopy category
$\Ho$ or in the stable homotopy category $\Ho^{\rm s}$.
Yet, the categories of simplicial sets (or $k$\nobreakdash-spaces) and symmetric spectra also admit monoidal model structures that
allow one to define monoids and modules within the model categories themselves.
In the rest of the article, we study the preservation of such \emph{strict} structures in monoidal model categories,
by viewing them as algebras over coloured operads and using suitable rectification functors.

Thus, from now on, we restrict ourselves to coloured operads
in simplicial sets (or $k$\nobreakdash-spaces with the Quillen model structure)
acting on simplicial (or topological) monoidal model categories.
A monoidal model category $\M$ is called \emph{simplicial} if it is also
a simplicial model category, and the simplicial action commutes with the
monoidal product, i.e., there are natural coherent isomorphisms
\[
K\boxtimes(X\otimes Y)\cong (K\boxtimes X)\otimes Y
\]
where $K$ is any simplicial set and $X$, $Y$ are objects of~$\M$.
The same definition applies to the topological case.

While all simplicial sets are cofibrant, this is not so for $k$\nobreakdash-spaces.
Therefore, cofibrancy assumptions will be needed at certain places.
A remedy would be to use $k$\nobreakdash-spaces with the Str{\o}m structure.
However, this model structure is not known to be cofibrantly generated; hence,
it does not fit into the framework described in the preceding section.
While it is still possible to talk of cofibrant operads and cofibrant
algebras in this setting (see \cite{Vog03}) and our results remain valid
with the same proofs, to avoid working in two different settings
we stick to the Quillen model structure whenever $k$\nobreakdash-spaces are considered.

For the sake of clarity, we will emphasize notationally the distinction between the monoidal model
category $\E$ in which our coloured operads take values and~the monoidal model category $\M$
on which they act.
Thus, if $P$ is a $C$\nobreakdash-coloured operad in the category $\E$ of simplicial sets
(or $k$\nobreakdash-spaces) and $\M$ is a simplicial (resp.\ topological) monoidal
model category, then a $P$\nobreakdash-algebra ${\bf X}=(X(c))_{c\in C}$ is defined as an object of $\M^C$
equipped with a morphism of $C$\nobreakdash-coloured operads 
$P\longrightarrow {\rm End}({\bf X})$ in~$\E$,
where ${\rm End}({\bf X})$ is now defined as
\begin{equation}
\label{simplicialaction}
{\rm End}({\bf X})(c_1,\ldots, c_n;c)=\Map(X(c_1)\otimes\cdots\otimes X(c_n), X(c)),
\end{equation}
and $\Map(-,-)$ denotes the simplicial (resp.\ topological) enrichment of~$\M$.
This is consistent with the previous definitions, since the map
\[
\xymatrix{
\sSets(P(c_1,\ldots, c_n;c), \Map(X(c_1)\otimes\cdots\otimes X(c_n), X(c))) \ar[d] \\
\M(P(c_1,\ldots, c_n; c)\boxtimes (X(c_1)\otimes\cdots\otimes X(c_n)), X(c))
}
\]
is bijective by the adjunction (\ref{adj}).

Simplicial sets or $k$\nobreakdash-spaces ${\rm Hom}_{P}({\bf X},{\bf Y})(c_1,\ldots, c_n;c)$
and ${\rm End}_{P}({\bf X})(c_1,\ldots, c_n;c)$ are defined analogously as in (\ref{simplicialaction})
and (\ref{restriction}), for every $C$\nobreakdash-coloured operad $P$.

We say that two $P$\nobreakdash-algebra structures $\gamma,\gamma'\colon P\longrightarrow {\rm End}_P({\bf X})$
on an object ${\bf X}$
of ${\M}^C$ coincide up to homotopy if $\gamma\simeq\gamma'$ in the model category of $C$\nobreakdash-coloured operads.
(Homotopic means left and right homotopic.)

The following is the main theorem of this article:

\begin{thm}
\label{mainthm}
Let $(L,\eta)$ be a homotopical localization on a simplicial or topological monoidal model category $\M$.
Let $P$ be a cofibrant $C$\nobreakdash-coloured operad in simplicial sets or $k$\nobreakdash-spaces, where $C$ is any set,
and consider the extension of $(L,\eta)$ over $\M^C$ away from an ideal $J\subseteq C$ relative to $P$.
Let ${\bf X}$ be a $P$\nobreakdash-algebra such that $X(c)$ is cofibrant in~$\M$ for every $c\in C$.
Suppose that $(\eta_{\bf X})_{c_1}\otimes\cdots\otimes (\eta_{\bf X})_{c_n}$
is an $L$\nobreakdash-equivalence whenever $P(c_1,\ldots,c_n;c)$ is nonempty.
Then $L{\bf X}$ admits a homotopy unique $P$\nobreakdash-algebra structure such that
$\eta_{{\bf X}}$ is a map of $P$\nobreakdash-algebras.
\end{thm}
\begin{proof}
We first check that the morphism of $C$\nobreakdash-coloured collections
\[
\End\nolimits_{P}(L{\bf X})\longrightarrow \Hom\nolimits_{P}({\bf X}, L{\bf X})
\]
induced by $\eta_{{\bf X}}$ is a trivial fibration, i.e., a trivial fibration of
simplicial sets or $k$\nobreakdash-spaces for every
$(c_1,\ldots,c_n;c)$. Since the value of both these $C$\nobreakdash-coloured collections on
$(c_1,\ldots,c_n; c)$ is the empty set whenever $P(c_1,\ldots,c_n; c)$ is the empty set, we may exclude
these cases from the argument.
Now, for all $(c_1,\ldots, c_n;c)$ such that $P(c_1,\ldots,c_n;c)$ is nonempty, the map
\[
(\eta_{\bf X})_{c_1}\otimes\cdots\otimes (\eta_{\bf X})_{c_n}\colon
X(c_1)\otimes\cdots\otimes X(c_n)\longrightarrow L_{c_1}X(c_1)\otimes\cdots\otimes L_{c_n}X(c_n)
\]
is an $L$\nobreakdash-equivalence by assumption.
It is also a cofibration, since, in any monoidal model category,
the tensor product of two cofibrations with cofibrant domains is a cofibration.
Hence, the map
\[
\xymatrix{
\Map(L_{c_1}X(c_1)\otimes\cdots\otimes L_{c_n}X(c_n), L_cX(c))\ar[d] \\
\Map(X(c_1)\otimes\cdots\otimes X(c_n), L_cX(c))
}
\]
is a trivial fibration. Indeed, if $c\in J$ then $c_i\in J$ for all $i$ (since $J$
is an ideal) and therefore the map is the identity; and if $c\not\in J$,
then it is a weak equivalence since $L_cX(c)=LX(c)$ is $L$\nobreakdash-local, and it is a
fibration by Quillen's axiom~SM7. The $P$\nobreakdash-algebra structure on $L{\bf X}$ is now
obtained similarly as in Lemma \ref{lem01}, as follows.
Consider the $C$\nobreakdash-coloured operad ${\rm End}_{P}(\eta_{{\bf X}})$, obtained as the following pullback of
$C$\nobreakdash-coloured collections:
\begin{equation}
\label{mainpullback}
\xymatrix{
{\rm End}_{P}(\eta_{{\bf X}}) \ar@{.>}[r]^{\rho} \ar@{.>}[d]_{\tau} & {\rm End}_{P}(L{\bf X}) \ar[d] \\
{\rm End}_{P}({\bf X}) \ar[r] & {\rm Hom}_{P}({\bf X}, L{\bf X}).
}
\end{equation}
The morphism $\tau$ is a trivial fibration since
it is a pullback of a trivial fibration, and the coloured operad $P$ is cofibrant by hypothesis.
Hence there is a lifting
\[
\xymatrix{
& {\rm End}_{P}(\eta_{{\bf X}}) \ar[d] \\
P \ar[r] \ar@{.>}[ur] & {\rm End}_{P}({\bf X})
}
\]
where $P\longrightarrow {\rm End}_{P}({\bf X})$ is the given $P$\nobreakdash-algebra structure of ${\bf X}$. Now, composing
this lifting with the upper morphism $\rho$ in~(\ref{mainpullback})
gives a $P$\nobreakdash-algebra structure on $L{\bf X}$ such that $\eta_{{\bf X}}$ is a map of $P$\nobreakdash-algebras, as claimed.

For the uniqueness, suppose that we have two $P$\nobreakdash-algebra structures on $L{\bf X}$, which we denote by
$\gamma,\gamma'\colon P\longrightarrow{\rm End}_{P}(L{\bf X})$, and assume further that $\eta_{\bf X}$
is a map of $P$\nobreakdash-algebras for each
of them, meaning that $\gamma$ and $\gamma'$ factor through ${\rm End}_{P}(\eta_{\bf X})$.
Thus, let $\delta,\delta'\colon P\longrightarrow {\rm End}_{P}(\eta_{\bf X})$ be such that
$\gamma=\rho\circ\delta$ and $\gamma'=\rho\circ\delta'$, where $\rho$ is the upper morphism
in~(\ref{mainpullback}).
Since $\tau\circ\delta=\tau\circ\delta'$ and $\tau$ is a trivial fibration, it follows that
$\delta$ and $\delta'$ are left homotopic.
Since $P$ is cofibrant and ${\rm End}_{P}(L{\bf X})$ is fibrant, we obtain that, in fact,
$\gamma\simeq\gamma'$; see \cite[7.4.8]{Hir03}.
\end{proof}

This result also holds if the $C$\nobreakdash-coloured operad $P$ is non\nobreakdash-symmetric.
In this case, we can replace it by its symmetric version $\Sigma P$, since both yield the same class of algebras
(see Remark \ref{nonsymmetric}).

Moreover, Theorem~\ref{mainthm} is also true for topological $C$\nobreakdash-coloured operads without the assumption that
they admit a model structure (e.g., if one uses the Str{\o}m model category structure on $k$\nobreakdash-spaces).
For this, one has to assume that the $C$\nobreakdash-coloured operad $P$ given in the statement of
Theorem~\ref{mainthm} is ``cofibrant'' in the sense that it has the left lifting
property with respect to morphisms of $C$\nobreakdash-coloured operads yielding trivial
fibrations of spaces at each tuple of colours.

The assumption that $P$ is cofibrant as a $C$\nobreakdash-coloured operad is essential in
the proof of Theorem~\ref{mainthm}.
In order to obtain a similar result for arbitrary coloured operads,
one needs that the monoidal model category $\M$ allows rectification
of algebras over resolutions of coloured operads. According to \cite{EM06},
this holds when $\M$ is the category of symmetric spectra. We will use this fact
in Section~\ref{rectifications} to extend Theorem~\ref{mainthm} in the case of spectra.

\subsection{$A_{\infty}$\nobreakdash-spaces and $E_{\infty}$\nobreakdash-spaces}
\label{morespaces}
Let us specialize to the model category of simplicial sets acting on itself. Let $P$ be any
$C$\nobreakdash-coloured operad in simplicial sets and choose a cofibrant resolution $P_{\infty}\longrightarrow P$.
If $(L,\eta)$ is any homotopical localization, then
the product of any two $L$\nobreakdash-equivalences is an $L$\nobreakdash-equivalence by the argument used in~(\ref{shift}).
Hence, we can apply Theorem \ref{mainthm}. Therefore, if
${\bf X}=(X(c))_{c\in C}$ is any $P_{\infty}$\nobreakdash-algebra, then
$L{\bf X}$ is again a $P_{\infty}$\nobreakdash-algebra and
$\eta_{\bf X}\colon {\bf X}\longrightarrow L{\bf X}$ is a map of $P_{\infty}$\nobreakdash-algebras.
The same statements are true in the category of $k$\nobreakdash-spaces, although in this case
we need suitable cofibrancy assumptions on the spaces $X(c)$ for $c\in C$ for the validity of the argument,
if the Quillen model structure is used.

In particular, this result applies to the operads $\Ass$ and $\Com$,
yielding the following result. Recall that an \textit{$A_{\infty}$\nobreakdash-space} is
an algebra over an arbitrary but fixed cofibrant resolution of $\Ass$, and an \textit{$E_{\infty}$\nobreakdash-space}
is an algebra over a cofibrant resolution of $\Com$.
Analogously, as defined in Subsection~\ref{pinftymaps},
by an \textit{$A_{\infty}$\nobreakdash-map} we mean an algebra over a cofibrant resolution
of $\Mor_{\Ass}$, and by an \textit{$E_{\infty}$\nobreakdash-map} we mean an algebra over a cofibrant resolution
of $\Mor_{\Com}$.

\begin{cor}
Let $(L,\eta)$ be a homotopical localization on the category of simplicial sets
or $k$\nobreakdash-spaces.
If $X$ is a cofibrant $A_{\infty}$\nobreakdash-space, then $LX$ has a homotopy unique $A_{\infty}$\nobreakdash-space
structure such that $\eta_X\colon X\longrightarrow LX$ is a map of $A_{\infty}$\nobreakdash-spaces.
Moreover, if $g$ is an $A_{\infty}$\nobreakdash-map between cofibrant $A_{\infty}$\nobreakdash-spaces,
then $Lg$ is also an $A_{\infty}$\nobreakdash-map.
The same statements are true for $E_{\infty}$\nobreakdash-spaces and $E_{\infty}$\nobreakdash-maps.
$\hfill\qed$
\label{ainfinityspc}
\end{cor}

As explained in the Introduction, the following is a consequence of Corollary~\ref{ainfinityspc},
using the fact that every loop space is an $A_{\infty}$\nobreakdash-space, and, conversely, every $A_{\infty}$\nobreakdash-space $X$
for which the monoid of connected components $\pi_0(X)$ is a group
is weakly equivalent to a loop space,
namely $\Omega BX$, where $B$ denotes the classifying space functor;
cf.\ \cite{Sta63}, \cite[Theorem~1.26]{BV73}, \cite{May74}.

\begin{cor}
If $(L,\eta)$ is a homotopical localization on the category of simplicial sets
or $k$\nobreakdash-spaces, and $X$ is a loop space, then $LX$
is naturally weakly equivalent to a loop space and the localization map
$\eta_X\colon X\longrightarrow LX$ is naturally weakly equivalent to a loop~map.
Moreover, if $g\colon X\longrightarrow Y$ is a loop map between loop spaces,
then $Lg$ is naturally weakly equivalent to a loop map.
\end{cor}

\begin{proof}
Let $Q$ be a cofibrant replacement functor, so that $QX\longrightarrow X$
is a trivial fibration and $QX$ is cofibrant. Here $X$ is an $A_{\infty}$\nobreakdash-space and,
by homotopy invariance, $QX$ is also an $A_{\infty}$\nobreakdash-space (see \cite[Theorem~4.58]{BV73}, \cite[Theorem~3.5.b]{BM03}).
Therefore, by~Corollary~\ref{ainfinityspc}, $LQX$ is an $A_{\infty}$\nobreakdash-space and $\eta_{QX}$ is a map of $A_{\infty}$\nobreakdash-spaces.
Since $\pi_0(LQX)\cong\pi_0(X)$ is a group, we may apply the classifying space functor, hence obtaining a commutative diagram
\begin{equation}
\label{locloop}
\xymatrix{
X \ar[d]_{\eta_X} & QX \ar[l]_{\simeq} \ar[d]^{\eta_{QX}} \ar[r]^{\simeq} & \Omega BQX \ar[d]^{\Omega B\eta_{QX}} \\
LX & LQX \ar[l]_{\simeq} \ar[r]^{\simeq} & \Omega BLQX.
}
\end{equation}

To prove the third claim, view $g$ as an $A_{\infty}$\nobreakdash-map. Then $Qg$ is also an $A_{\infty}$\nobreakdash-map,
and, by Corollary~\ref{ainfinityspc}, $LQg$ is an $A_{\infty}$\nobreakdash-map, hence
naturally weakly equivalent (as a functor on~$g$) to a loop map between loop spaces.
\end{proof}

Essentially the same result was obtained in
\cite[Theorem~3.1]{Bou94} for nullifications and in \cite[Lemma~A.3]{Far96}
for $f$\nobreakdash-localizations, using Segal's theory of loop spaces.
As we next show, their delooping of $L_f\Omega$
coincides, up to homotopy, with the one given by~(\ref{locloop}).

\begin{prop}
\label{flocalizationofloops}
Let $f$ be any map. Then $L_f\Omega Y\simeq \Omega L_{\Sigma f} Y$ for all~$Y$.
\end{prop}

\begin{proof}
It follows from (\ref{locloop}) that $L_f\Omega Y\simeq \Omega FY$,
where $F$ is a functor, namely $F=BL_fQ\Omega$.
Note that there is a natural transformation
$\zeta\colon BQ\Omega\longrightarrow F$ and there is also a natural isomorphism
$\xi\colon BQ\Omega\longrightarrow {\rm Id}$ on the
homotopy category of connected spaces. It follows that, if $\lambda=\zeta\circ\xi^{-1}$,
then $(F,\lambda)$ is a localization on this category. On one hand, a connected
space $Y$ is $F$\nobreakdash-local if and only if
$\Omega Y$ is $L_f$\nobreakdash-local. On the other hand,
$\Omega Y$ is $L_f$\nobreakdash-local if and only if it
is simplicially orthogonal to~$f$, and this happens if and only if $Y$ is simplicially orthogonal
to~$\Sigma f$. Hence, $F$ and $L_{\Sigma f}$ are localizations on the same
category with the same class of local objects, from which it follows that
there is a homotopy equivalence $FY\simeq L_{\Sigma f}Y$ under $Y$,
for all connected spaces~$Y$.

If $Y$ is not connected, then we take
the basepoint component $Y_0$ and have
\[
L_f\Omega Y=L_f\Omega Y_0\simeq \Omega L_{\Sigma f}Y_0=\Omega (L_{\Sigma f}Y)_0=\Omega L_{\Sigma f}Y,
\]
hence completing the proof.
\end{proof}

The preservation of infinite loop spaces and infinite loop maps under homotopical
localizations follows either iteratively or by repeating the above arguments
with $E_{\infty}$ instead of $A_{\infty}$.

\subsection{$A_{\infty}$\nobreakdash-structures and $E_{\infty}$\nobreakdash-structures on spectra}
\label{nowspectra}
Now let $\M$ be the category of symmetric spectra over simplicial sets.
In order to handle commutative ring spectra and their modules conveniently,
we endow $\M$ with the {\it positive} stable model structure,
which was discussed in \cite{MMSS01}, \cite{Sch01}, or \cite{Shi04}.
Thus, weak equivalences in $\M$ are the usual stable weak equivalences (as defined in \cite{HSS00}),
and positive cofibrations are stable cofibrations as in \cite{HSS00} with the additional
assumption that they~are isomorphisms in level zero. Positive fibrations are defined
by the right lifting property with respect to the trivial positive cofibrations.
By \cite[Proposition 3.1]{Shi04},
the category of symmetric spectra over simplicial sets with the positive stable
model structure is a cofibrantly generated, proper, monoidal model category,
and so is the category of $R$\nobreakdash-modules for every ring spectrum~$R$.

A spectrum $X$ is called {\it connective}
if it is $(-1)$\nobreakdash-connected, i.e., $\pi_k(X)=0$ for $k<0$.
If $\map(-,-)$ is any homotopy function complex in $\M$, then
\[
\pi_n\map(X,Y)\cong [\Sigma^nX, Y]\cong \pi_n F(X,Y)
\]
for all spectra $X$, $Y$ and $n\ge 0$, where $F(-,-)$ denotes
the derived function spectrum; cf.\ \cite[Lemma~6.1.2]{Hov99}.
In other words, the simplicial
set $\map(X,Y)$ has the same homotopy groups (with any choice of a basepoint) as the
connective cover $F^c(X,Y)$ of the spectrum $F(X,Y)$.
Hence, if $L$ is a homotopical localization on~$\M$, then
a map $f\colon X\longrightarrow Y$ is an $L$\nobreakdash-equivalence if and only if
\begin{equation}
\label{Fc}
F^c(f,Z)\colon F^c(Y,Z)\longrightarrow F^c(X,Z)
\end{equation}
is a weak equivalence of spectra for every $L$\nobreakdash-local spectrum $Z$.

As a consequence of this fact, \emph{the smash product of two $L$\nobreakdash-equivalences
need not be an $L$\nobreakdash-equivalence}, but it
is so if any one of the sufficient conditions stated in Theorem~\ref{f1smashf2}
is satisfied; cf.\ \cite{CG05}.

We say that the functor $L$ \emph{commutes with suspension\/} if
$L\Sigma X\simeq \Sigma LX$ for all~$X$.
Note that, by (\ref{Fc}), if $f$ is any $L$\nobreakdash-equivalence,
then so is~$\Sigma f$. Therefore, for every spectrum~$X$, the map
$\Sigma\eta_X\colon \Sigma X\longrightarrow \Sigma LX$ is an $L$\nobreakdash-equivalence.
For $X$ cofibrant, this yields a map (in fact, an $L$\nobreakdash-equivalence)
\[
g_X\colon\Sigma LX\longrightarrow L\Sigma X,
\]
unique up to homotopy, such that $g_X\circ \Sigma\eta_X \simeq \eta_{\Sigma X}$.
It is natural to say that $L$ commutes with suspension if $g_X$ is a weak
equivalence for all (cofibrant)~$X$. However, this is equivalent to the condition that
$\Sigma LX$ be weakly equivalent to an $L$\nobreakdash-local spectrum for all~$X$, and hence to the condition that 
$L\Sigma X\simeq \Sigma LX$ for all~$X$.

If $L$ commutes with suspension and $Z$ is $L$\nobreakdash-local, then (a fibrant replacement of)
$\Sigma^n Z$ is also $L$\nobreakdash-local, not only for $n\le 0$, but also for $n>0$.
From this fact it follows that a map $f$ is an $L$\nobreakdash-equivalence
if and only if $F(f,Z)$ is a weak equivalence for every $L$\nobreakdash-local spectrum $Z$.
Hence, the condition that $L$ commutes with suspension holds if
and only if $L$ is \emph{closed} on $\Ho(\M)$ in the sense of Section~\ref{preservation} above.

\begin{thm}
\label{f1smashf2}
Let $(L,\eta)$ be a homotopical localization on symmetric spectra.
Let $f_1\colon X_1\longrightarrow Y_1$ and $f_2\colon X_2\longrightarrow Y_2$
be $L$\nobreakdash-equivalences.
Suppose that any one of the following conditions is satisfied:
\begin{itemize}
\item[{\rm (i)}] $L$ commutes with suspension.
\item[{\rm (ii)}] $X_1$ and $Y_2$ are connective.
\item[{\rm (iii)}] $f_1$ is a weak equivalence between connective spectra.
\end{itemize}
Then the derived smash product $f_1\wedge f_2\colon X_1\wedge X_2\longrightarrow Y_1\wedge Y_2$ 
is an $L$\nobreakdash-equivalence.
\end{thm}

\begin{proof}
If $L$ commutes with suspension, then $L$ is closed on $\Ho(\M)$ and
therefore we may use the same argument as in (\ref{shift}).

Now assume that the spectra $X_1$ and $Y_2$ are connective. The following argument
is due to Bousfield \cite{Bou99}. One first proves
that $X_1\wedge f_2$ is an $L$\nobreakdash-equivalence as follows. If $Z$ is any $L$\nobreakdash-local
spectrum, then
\[
F^c(X_1\wedge f_2,Z)\simeq F^c(X_1,F(f_2,Z))\simeq F^c(X_1,F^c(f_2,Z))
\]
since $X_1$ is connective. Here $F^c(f_2,Z)$ is a weak equivalence and this implies that
$F^c(X_1\wedge f_2,Z)$ is also a weak equivalence. Then the same method proves that $f_1\wedge Y_2$
is an $L$\nobreakdash-equivalence. Finally, $f_1\wedge f_2=(f_1\wedge Y_2)\circ (X_1\wedge f_2)$,
and the argument is complete. 
Of course, the same argument is valid if, instead, $X_2$ and $Y_1$ are connective.
Moreover, if $f_1$ is a weak equivalence, then
we only need that $X_1$ be connective, since $f_1\wedge Y_2$ is in this case a weak equivalence,
and similarly if the indices are exchanged.
\end{proof}

We emphasize that this apparently \emph{ad hoc} result is crucial in the proof of
Corollary~\ref{monster}, where connectivity conditions are imposed
in the case when $L$ does not commute with suspension.
These connectivity conditions are justified by the result that we have just shown,
and their necessity will be demonstrated with counterexamples
at the end of this section.

Let us recall that an \textit{$A_{\infty}$\nobreakdash-ring} is an algebra over a cofibrant resolution of $\Ass$
(which need therefore not be a strict ring, although it is weakly equivalent to one).
An \textit{$A_{\infty}$\nobreakdash-map of $A_{\infty}$\nobreakdash-rings}
is an algebra over a cofibrant resolution of $\Mor_{\Ass}$ (which, as explained
in Subsection~\ref{pinftymaps}, is a weaker notion than a morphism
of $A_{\infty}$\nobreakdash-rings). If $(R,M)$ is an algebra over a cofibrant resolution of $\LMod_{\Ass}$,
then $M$ is called a \textit{left $A_{\infty}$\nobreakdash-module} over~$R$, as in Subsection~\ref{ainftymods}.
Accordingly, an \textit{$A_{\infty}$\nobreakdash-map of left $A_{\infty}$\nobreakdash-modules} is an algebra over
a cofibrant resolution of $\Mor_{P}$ where $P=\LMod_{\Ass}$. The same terminology is used with $E_{\infty}$.

Note that, if the value of a $C$\nobreakdash-coloured operad $P$ on a given tuple of colours $(c_1,\ldots,c_n;c)$
is the empty set, and $P_{\infty}\longrightarrow P$
is a cofibrant resolution, then the value of $P_{\infty}$ on $(c_1,\ldots,c_n;c)$ is also the empty set, since
\[ P_{\infty}(c_1,\ldots,c_n;c)\longrightarrow P(c_1,\ldots,c_n;c) \] is a weak equivalence.
This ensures that, if $J$ is an ideal relative to $P$, then $J$ is also an ideal relative to $P_{\infty}$.
This fact is important for the validity of the next result.

\begin{cor}
\label{monster}
Let $(L,\eta)$ be a homotopical localization on symmetric spectra that commutes with suspension.
Let $M$ be a left $A_{\infty}$\nobreakdash-module over an $A_{\infty}$\nobreakdash-ring $R$, and assume that both $R$
and $M$ are cofibrant as spectra. Then the following hold:
\begin{itemize}
\item[{\rm (i)}]
$LR$ has a homotopy unique $A_{\infty}$\nobreakdash-ring structure such that $\eta_R\colon R\longrightarrow LR$
is a morphism of $A_{\infty}$\nobreakdash-rings.
\item[{\rm (ii)}]
$LM$ has a homotopy unique left $A_{\infty}$\nobreakdash-module structure over $R$ such that $\eta_M\colon M\longrightarrow LM$ is
a morphism of $A_{\infty}$\nobreakdash-modules.
\item[{\rm (iii)}]
$LM$ admits a homotopy unique left $A_{\infty}$\nobreakdash-module structure over $LR$ extending the 
$A_{\infty}$\nobreakdash-module structure over~$R$.
\item[{\rm (iv)}]
If $f\colon R\longrightarrow T$ is an $A_{\infty}$\nobreakdash-map of cofibrant $A_{\infty}$\nobreakdash-rings, then $Lf$ admits
a homotopy unique compatible $A_{\infty}$\nobreakdash-map structure.
\item[{\rm (v)}]
If $g\colon M\longrightarrow N$ is an $A_{\infty}$\nobreakdash-map of cofibrant left $A_{\infty}$\nobreakdash-modules over~$R$,
then $Lg$ admits a homotopy unique compatible structure of an 
$A_{\infty}$\nobreakdash-map of left $A_{\infty}$\nobreakdash-modules over~$R$, and also over~$LR$.
\end{itemize}
\end{cor}

Similar statements are true for right modules and bimodules,
and the same results hold for $E_{\infty}$\nobreakdash-rings and their modules.

If $L$ does not commute with suspension, then the same statements hold by assuming
that $R$ and $LR$ are connective in~{\rm (i)}; that $R$ is connective in~{\rm (ii)}; that $R$ and $LR$ are
connective, and at least one of $M$ and $LM$ is connective in~{\rm (iii)}; that $R$, $T$, $LR$, and $LT$
are connective in~{\rm (iv)}; that $R$ is connective for the first claim in~{\rm (v)}, and
that $R$, $LR$, $M$ or $LM$, and $N$ or $LN$ are connective for the second claim in~{\rm (v)}.

\begin{proof}
In part (i), use a cofibrant resolution of the operad $\Ass$.
If $L$ commutes with suspension, then the result follows from Theorem \ref{mainthm}, since
every finite smash product of $L$\nobreakdash-equivalences is an $L$\nobreakdash-equivalence. If $L$ does not commute with suspension,
then we need
to prove that $\eta_R\wedge\cdots\wedge\eta_R$ is an $L$\nobreakdash-equivalence for any finite number of factors. 
By Theorem~\ref{f1smashf2}, this follows
from the fact that $R$ and $LR$ are connective.

In part (ii), use a cofibrant resolution of the non\nobreakdash-symmetric $C$\nobreakdash-coloured operad $\LMod_{\Ass}$ with $C=\{r,m\}$
described in Subsection~\ref{mulmod},
and choose the ideal \hbox{$J=\{r\}$}. Thus, $(R,M)$ is an algebra over this $C$\nobreakdash-coloured operad.
In order to prove that $R\wedge\cdots\wedge R\wedge\eta_M$ is an $L$\nobreakdash-equivalence (where $R$
appears an arbitrary number of times, while $\eta_M$ appears precisely once), we only need that
$R$ be connective.

To prove (iii), use again a cofibrant resolution of $\LMod_{\Ass}$ with $C=\{r,m\}$,
and choose the ideal $J=\emptyset$. Here we need that $\eta_R\wedge\cdots\wedge\eta_R$ be an $L$\nobreakdash-equivalence
for any number of factors,
which is the case if either $L$ commutes with suspension or $R$ and $LR$ are connective, and we also need that $\eta_R\wedge\cdots\wedge\eta_R\wedge\eta_M$ be an $L$\nobreakdash-equivalence for any number of factors, where
$\eta_M$ appears precisely once. This is the case if either $L$ commutes with suspension, or $R$ and $LR$
and at least one of $M$ and $LM$ are connective.

For part~(iv), use a cofibrant resolution of the coloured operad $\Mor_{\Ass}$ with $J=\emptyset$.
We need that $\eta_R\wedge\cdots\wedge\eta_R\wedge\eta_T\wedge\cdots\wedge\eta_T$ be an $L$\nobreakdash-equivalence
for any number of factors, which happens if either $L$ commutes with suspension or $R$, $LR$,
$T$, and $LT$ are connective.

Similarly, in part~(v)
use a cofibrant resolution of the coloured operad $\Mor_Q$ with $Q=\LMod_{\Ass}$.
In order to infer that $Lg$ is an $A_{\infty}$\nobreakdash-map of $A_{\infty}$\nobreakdash-modules
over~$R$, choose the ideal $J=\{(0,r),(1,r)\}$. In the case when $L$ does not commute with suspension,
it is enough to assume that $R$ be connective.
If we wish to infer that $Lg$ is an $A_{\infty}$\nobreakdash-map of $A_{\infty}$\nobreakdash-modules over $LR$,
then we have to choose $J=\emptyset$, and, if $L$ does not commute with suspension,
we need to add the assumption that $LR$ be connective and at least one of $M$ and $LM$
be connective and furthermore at least one of $N$ and $LN$ be connective,
by the same reason as in part~(iii).
\end{proof}

If $L$ does not commute with suspension, then the condition that $R$ be connective
cannot be dropped in part~(i). Indeed, the $n$th Postnikov section functor $P_n$ is a
homotopical localization for all~$n$, and, if $R$ is nonconnective, then $P_{-1}R$ does not admit
a ring spectrum structure ---not even up to homotopy--- since the composite
of the unit map $\nu\colon S\longrightarrow P_{-1}R$ with the multiplication map
\[
S\wedge P_{-1}R \longrightarrow P_{-1}R \wedge P_{-1}R \longrightarrow P_{-1}R
\]
has to be a homotopy equivalence, yet $\nu$ is null since $\pi_0(P_{-1}R)=0$.

Similarly, in part~(ii), we need that $R$ be connective, since otherwise the Postnikov sections of $R$
need not be homotopy modules over~$R$. The following example is a simpler version of
\cite[Example~4.4]{CG05}. Let $K(n)$ be the Morava $K$\nobreakdash-theory spectrum for any prime $p$ and $n\ge 1$,
and let $i$ be any integer. If $P_iK(n)$ were a homotopy module spectrum over $K(n)$, then the composite
of the unit map of $K(n)$ with the structure map of $P_iK(n)$
\begin{equation}
\label{PKn}
S\wedge P_iK(n) \longrightarrow K(n)\wedge P_iK(n) \longrightarrow P_iK(n)
\end{equation}
would be a homotopy equivalence. However, $K(n)\wedge H\Z/p\simeq 0$ while $P_iK(n)$
has nonzero mod~$p$ homology; see~\cite[p.~545]{Rud98}.

In part (iii), we need in addition
that either $M$ or $LM$ be connective; otherwise a counterexample can be displayed as follows.
Let $R$ be the integral Eilenberg\nobreakdash--Mac Lane spectrum $H\Z$ and let $L$ be localization with
respect to the map $f\colon S\longrightarrow S\Q$, where $S\Q$ denotes
a rational Moore spectrum, and the map $f$ is induced by the inclusion $\Z\hookrightarrow\Q$.
Then $LR\simeq H\Q$.
However, if $M=\Sigma^{-1}H\Z$, then $M$ is $L$\nobreakdash-local, yet it is not an $H\Q$\nobreakdash-module.
Incidentally, this example shows that the condition that either $M$ or $LM$ be connective
was also necessary in \cite[Theorem 4.5]{CG05}.

\section{Rectification results for spectra}
\label{rectifications}
Let $\M$ be, as above, the category of symmetric spectra over simplicial sets
with the positive stable model structure.
According to \cite[Theorem 1.3]{EM06}, for every set $C$ and every $C$\nobreakdash-coloured operad $P$
in simplicial sets, there is a model structure on the category of $P$\nobreakdash-algebras in $\M$
in which a map of $P$\nobreakdash-algebras ${\bf X}\longrightarrow {\bf Y}$ is a weak equivalence
(resp.\ a fibration) if and only if,
for each $c\in C$, the map $X(c)\longrightarrow Y(c)$ is a weak equivalence
(resp.\ a positive fibration) of symmetric spectra.
If $P=\Ass$ or $P=\Com$, then the corresponding model structures coincide with the model
structures used in categories of ring spectra by other authors, e.g. in \cite{Shi04}.

If $P$ is a $C$\nobreakdash-coloured operad in simplicial sets and $\varphi\colon P_{\infty}\longrightarrow P$
is a cofibrant resolution, then it follows from \cite[Theorem 1.4]{EM06} that the adjoint pair
\begin{equation}
\label{quillenpair}
\varphi_! : {\rm Alg}_{P_{\infty}}(\M)\rightleftarrows {\rm Alg}_{P}(\M) : \varphi^*,
\end{equation}
where $\varphi^*$ assigns to each $P$\nobreakdash-algebra the $P_{\infty}$\nobreakdash-algebra structure given by composing with~$\varphi$,
and $\varphi_!$ is its left adjoint, defines a Quillen equivalence.
Consequently, if $\bf X$ is a $P_{\infty}$\nobreakdash-algebra,
and $Q$ is a cofibrant replacement functor on the model category of ${P_{\infty}}$\nobreakdash-algebras,
while $F$ is a fibrant replacement functor on the model category of $P$\nobreakdash-algebras,
then the unit map
\[
Q{\bf X}\longrightarrow \varphi^*F\varphi_{!}Q{\bf X}
\]
is a weak equivalence; see \cite[Corollary 1.3.16]{Hov99}.
Thus, $\varphi_{!}Q{\bf X}$ is a functorial \textit{rectification} of ${\bf X}$
for each $P_{\infty}$\nobreakdash-algebra ${\bf X}$. Indeed, $\varphi_{!}Q{\bf X}$ is a $P$\nobreakdash-algebra
and its component at $c$ is weakly equivalent to $X(c)$ in $\M$ for all $c\in C$.

Dually, if ${\bf X}$ is a $P$\nobreakdash-algebra, then the counit map
$
\varphi_!Q\varphi^*F{\bf X} \longrightarrow F{\bf X}
$
is a weak equivalence of $P$\nobreakdash-algebras. 
Since weak equivalences of $P$\nobreakdash-algebras are defined componentwise, 
$\varphi^*$ preserves them. This implies that,
if ${\bf X}$ is a $P$\nobreakdash-algebra, then 
\begin{equation}
\label{counit}
\varphi_{!}Q\varphi^*{\bf X}\simeq \varphi_{!}Q\varphi^*F{\bf X}\simeq
F{\bf X} \simeq{\bf X}
\end{equation}
\emph{as $P$\nobreakdash-algebras} (thus, rectifying a $P_{\infty}$\nobreakdash-algebra which
is in fact a $P$\nobreakdash-algebra yields a
weakly equivalent $P$\nobreakdash-algebra). This will be relevant in the rest of the article.

We label the following statements for subsequent reference.

\begin{lem}
\label{rectifying}
Let ${\bf X}$ and ${\bf Y}$ be $P$\nobreakdash-algebras in $\M$, where $P$ is a $C$\nobreakdash-coloured operad.
Let $\varphi\colon P_{\infty}\longrightarrow P$ be a cofibrant resolution.
If $\varphi^*{\bf X}$ and $\varphi^*{\bf Y}$ are weakly equivalent as $P_{\infty}$\nobreakdash-algebras, then ${\bf X}$
and ${\bf Y}$ are weakly equivalent as $P$\nobreakdash-algebras.
\end{lem}

\begin{proof}
Let $Q$ be a cofibrant replacement functor on $P_{\infty}$\nobreakdash-algebras.
Since $\varphi_!$ preserves weak equivalences between cofibrant objects,
we have $\varphi_!Q\varphi^*{\bf X}\simeq\varphi_!Q\varphi^*{\bf Y}$
as $P$\nobreakdash-algebras, and it follows from~(\ref{counit}) that ${\bf X}\simeq{\bf Y}$, as claimed.
\end{proof}

\begin{lem}
\label{cofibrantcomponents}
If $P$ is a cofibrant $C$\nobreakdash-coloured operad in simplicial sets and $\bf X$ is a cofibrant $P$\nobreakdash-algebra
in~$\M$, then $X(c)$ is cofibrant for all $c\in C$.
\end{lem}

\begin{proof}
If $C$ has only one colour, this follows from
\cite[Proposition 4.3]{BM03} and \cite[Corollary 5.5]{BM03}.
The extension to several colours follows from the argument used
in the proof of \cite[Theorem 4.1]{BM07}.
\end{proof}

In the category of arrows of $\M^C$ we consider the model structure
whose weak equivalences and fibrations are componentwise. Thus,
two vertical arrows $f$ and $f'$ are weakly equivalent if there is a zig-zag
of commutative squares 
\[
\label{zigzagofmaps}
\xymatrix{
{\bf X} \ar[d]_{f} \ar[r]^{\simeq} & {\bf X}_0 \ar[d]^{f_0} & {\bf X}_1 \ar[d]^{f_1} \ar[r]^{\simeq} \ar[l]_{\simeq} &
\cdots & {\bf X}_n \ar[d]^{f_n} \ar[r]^{\simeq} \ar[l]_{\simeq} & {\bf X}' \ar[d]^{f'} \\
{\bf Y} \ar[r]^{\simeq} & {\bf Y}_0 & {\bf Y}_1 \ar[r]^{\simeq} \ar[l]_{\simeq} &
\cdots & {\bf Y}_n \ar[r]^{\simeq} \ar[l]_{\simeq} & {\bf Y}'
}
\]
whose horizontal arrows are weak equivalences at each colour.

We say that two functors $F$ and $F'$ from any given category to a model category
are \emph{naturally weakly equivalent} if there is a zig-zag
of natural transformations between $F$ and $F'$ that are weak equivalences at every object. 
For an object $X$, we will say that $FX$ and $F'X$ are
naturally weakly equivalent if $F$ and $F'$ are clear from the context
and naturally weakly equivalent.

The following result is inferred from Theorem~\ref{mainthm} and will yield
the main results in this section as special cases. To grasp its significance,
note that no cofibrancy assumption is made on the coloured operad~$P$.

\begin{thm}
\label{javierslemma}
Let $(L,\eta)$ be a homotopical localization on the model category $\M$ of symmetric spectra.
Let $P$ be a $C$\nobreakdash-coloured operad in simplicial sets, where $C$ is any set,
and consider the extension of $(L,\eta)$ over $\M^C$
away from an ideal $J\subseteq C$ relative to $P$.
Let ${\bf X}$ be a $P$\nobreakdash-algebra such that $X(c)$ is cofibrant for each $c\in C$,
and let $\eta_{\bf X}\colon {\bf X}\longrightarrow L{\bf X}$ be the localization map.
Suppose that $(\eta_{\bf X})_{c_1}\wedge\cdots\wedge (\eta_{\bf X})_{c_n}$
is an $L$\nobreakdash-equivalence whenever $P(c_1,\ldots,c_n;c)$ is nonempty.
Then there is a map $\xi_{\bf X}\colon D{\bf X}\longrightarrow T{\bf X}$ of $P$\nobreakdash-algebras, depending
functorially on ${\bf X}$, such that:
\begin{itemize}
\item[{\rm (i)}] ${\bf X}$ and $D{\bf X}$ are naturally weakly equivalent as $P$\nobreakdash-algebras;
\item[{\rm (ii)}] $L{\bf X}$ and $T{\bf X}$ are naturally
weakly equivalent as $P_{\infty}$\nobreakdash-algebras;
\item[{\rm (iii)}] $\eta_{\bf X}$ and $\xi_{\bf X}$ are naturally weakly equivalent as
$(\Mor_P)_{\infty}$\nobreakdash-algebras.
\end{itemize}
\end{thm}

\begin{proof}
Let $\varphi\colon P_{\infty}\longrightarrow P$ be a cofibrant resolution of~$P$, and
let $(\varphi_{!},\varphi^*)$ be the corresponding Quillen equivalence between the categories
of $P_{\infty}$\nobreakdash-algebras and $P$\nobreakdash-algebras, as in~(\ref{quillenpair}).
We view ${\bf X}$ as a $P_{\infty}$\nobreakdash-algebra via~$\varphi^*$.
Let $\Phi\colon (\Mor_P)_{\infty}\longrightarrow \Mor_P$ be a cofibrant resolution 
of~$\Mor_P$, and let
\[
\Phi_! : {\rm Alg}_{(\Mor_P)_{\infty}}(\M)\rightleftarrows {\rm Alg}_{\Mor_P}(\M) : \Phi^*
\]
be the corresponding adjoint pair.

Since $P_{\infty}(c_1,\ldots,c_n;c)$ is nonempty precisely when $P(c_1,\ldots,c_n;c)$ is nonempty,
it follows from Theorem~\ref{mainthm} that $L{\bf X}$ admits a $P_{\infty}$\nobreakdash-algebra
structure such that $\eta_{\bf X}$ is a map of $P_{\infty}$\nobreakdash-algebras.
Thus $\eta_{\bf X}$ is an algebra over $\Mor_{P_{\infty}}$, and it is also
an algebra over $\left(\Mor_P\right)_{\infty}$ using~(\ref{lifting}).
If $Q$ denotes a cofibrant replacement functor on $(\Mor_P)_{\infty}$\nobreakdash-algebras
and $F$ denotes a fibrant replacement functor on $\Mor_P$\nobreakdash-algebras, then
there are weak equivalences of $\left(\Mor_P\right)_{\infty}$\nobreakdash-algebras
\begin{equation}
\label{mainzigzag}
\xymatrix{
\eta_{\bf X} & \ar[l]_-{\simeq} Q\eta_{\bf X} \ar[r]^-{\simeq} & \Phi^*F\Phi_{!}Q\eta_{\bf X}.
}
\end{equation}
Hence, $\eta_{\bf X}$ is weakly equivalent to $\Phi^*\xi_{\bf X}$ as a $(\Mor_P)_{\infty}$\nobreakdash-algebra,
where
\[
\xi_{\bf X}=F\Phi_{!}Q\eta_{\bf X}.
\]
Note that $\xi_{\bf X}$ depends functorially on~${\bf X}$.
Hence, if we denote by $D{\bf X}$ the domain of $\xi_{\bf X}$
and by $T{\bf X}$ its target, then $D$ and $T$ are endofunctors
in the category of $P$\nobreakdash-algebras.

For $i\in\{0,1\}$, let $\alpha_i$ denote the inclusion of $C$ into $\{0,1\}\times C$ as $\alpha_i(c)=(i,c)$,
and choose a lifting $\lambda_i$ as in (\ref{jajajaja}),
\begin{equation}
\label{lambdai}
\xymatrix{
& (\alpha_i)^*(\Mor_P)_{\infty} \ar[d]^{(\alpha_i)^*\Phi} \\
P_{\infty} \ar@{.>}[ur]^{\lambda_i} \ar[r]^-{\varphi} & P.
}
\end{equation}
Now, if we apply $(\lambda_i)^*(\alpha_i)^*$ to~(\ref{mainzigzag}), we obtain weak
equivalences of $P_{\infty}$\nobreakdash-algebras. Let us choose first $i=0$. On one hand, using~(\ref{lambdai}),
\[
(\lambda_0)^*(\alpha_0)^*\Phi^*\xi_{\bf X}=
(\lambda_0)^*((\alpha_0)^*\Phi)^* D{\bf X}=
\varphi^*D{\bf X}.
\]
On the other hand, it follows from (\ref{fancy}) that
\[
(\lambda_0)^*(\alpha_0)^*\eta_{\bf X}=\varphi^*{\bf X}.
\]
Therefore, Lemma~\ref{rectifying} implies that ${\bf X}\simeq D{\bf X}$ as $P$\nobreakdash-algebras,
and the argument given in the proof of Lemma~\ref{rectifying} preserves naturality.
Secondly, for $i=1$ we obtain similarly weak equivalences of $P_{\infty}$\nobreakdash-algebras
\[
\xymatrix{
L\bf X & (\lambda_1)^*(\alpha_1)^*Q\eta_{\bf X} \ar[l]_-{\simeq} \ar[r]^-{\simeq} & \varphi^*T{\bf X},
}
\]
as claimed.
\end{proof}

\begin{rem}
If the assumption that $X(c)$ is cofibrant for all $c$ is not satisfied,
then $L{\bf X}$ need not be a $P_{\infty}$\nobreakdash-algebra and 
$\eta_{\bf X}$ need not be an algebra over $(\Mor_P)_{\infty}$.
In fact, Theorem~\ref{javierslemma} still holds, although we can only deduce that
$\eta_{\bf X}$ and $\xi_{\bf X}$ are naturally weakly equivalent as arrows in~$\M^C$,
and ${\bf X}\simeq D{\bf X}$ as $P$\nobreakdash-algebras.
To prove this, pick a functorial cofibrant replacement of $\bf X$
as a $P$\nobreakdash-algebra, ${\bf X}'\longrightarrow{\bf X}$.
By Lemma~\ref{cofibrantcomponents}, each $X'(c)$ is then cofibrant.
Therefore the argument proceeds for ${\bf X}'$ in the same way as above,
and we reach the conclusion that $\eta_{{\bf X}'}$ is 
naturally weakly equivalent as an algebra over $(\Mor_P)_{\infty}$ to a map
$\xi_{\bf X}\colon D{\bf X}\longrightarrow T{\bf X}$
of $P$\nobreakdash-algebras, still depending functorially on~$\bf X$, where $D{\bf X}$
is weakly equivalent to ${\bf X}'$ (and hence to~$\bf X$) as a $P$\nobreakdash-algebra.
Since $\eta_{\bf X}$ and $\eta_{{\bf X}'}$ are naturally weakly equivalent as arrows in~$\M^C$,
we have completed the argument.
\end{rem}

In summary, $L$ is weakly equivalent in $\M^C$ to a functor $T$ 
that sends $P$\nobreakdash-algebras to $P$\nobreakdash-algebras
(where $P$ is any coloured operad, not necessarily cofibrant).
However, $T$ is not coagumented, that is, there is no natural map ${\bf X}\longrightarrow
T{\bf X}$ in general.

Theorem~\ref{javierslemma} specializes to the following conclusive results. First we state the
preservation of (strict) ring spectra:

\begin{thm}
\label{strictringspectra}
Let $(L,\eta)$ be a homotopical localization on symmetric spectra. If~$R$ is a ring spectrum, then
$\eta_R\colon R\longrightarrow LR$ is
naturally weakly equivalent, as a map of spectra, 
to a ring morphism $\xi_R\colon DR\longrightarrow TR$,
provided that $L$ commutes with suspension or $R$ and $LR$ are connective.
Moreover, $DR\simeq R$ as ring spectra, and, if $R$ is commutative, then $DR$ and $TR$
can be chosen to be commutative.
\end{thm}

\begin{proof}
This follows from Theorem~\ref{javierslemma} by choosing $P=\Ass$ and $P=\Com$,
with $J$~empty in each case.
\end{proof}

An analogous result holds for $R$\nobreakdash-modules, as stated below. Here another subtlety
arises since, at a first attempt, localizing an $R$\nobreakdash-module will yield an $R'$\nobreakdash-module
where $R'$ is weakly equivalent to~$R$, although they are in principle distinct. This difficulty
is surmounted by means of the following remarks.

If $R$ is any ring spectrum, we endow the category of left $R$\nobreakdash-modules with the model structure of
\cite[Corollary 5.4.2]{HSS00}; that is, weak equivalences are $R$\nobreakdash-module morphisms that are weak
equivalences of the underlying spectra, and fibrations are $R$\nobreakdash-module morphisms that are positive
fibrations of the underlying spectra. This is coherent with the model structure that we are
considering on the category of $\LMod_{\Ass}$\nobreakdash-algebras, by associating each $R$\nobreakdash-module $M$
with the pair $(R,M)$.
If $R$ is commutative, we also consider the analogous
model structure on the category of $R$\nobreakdash-algebras, as given by \cite[Corollary 5.4.3]{HSS00}.

If $\rho\colon R\longrightarrow R'$ is a morphism of ring spectra, then restriction
sends every left $R'$\nobreakdash-module $M$ to the left $R$\nobreakdash-module $\rho^*M$
(which is the same spectrum $M$ with the module structure given by composition with~$\rho$),
and induction sends every left $R$\nobreakdash-module $N$ to the
left $R'$\nobreakdash-module $R'\wedge_R N$, where $R$ acts on $R'$ via~$\rho$.
It then follows that, if $\rho$ is a weak equivalence of ring spectra, then
the model categories of left $R$\nobreakdash-modules and left $R'$\nobreakdash-modules 
are Quillen equivalent via induction and restriction, by \cite[Theorem~5.4.5]{HSS00}.
More generally, the following holds:

\begin{lem}
\label{changeofrings}
Let $R$ and $R'$ be weakly equivalent ring spectra.
Then every left $R'$\nobreakdash-module $M$ is naturally weakly equivalent
as a spectrum to the $R$\nobreakdash-module
\[
R\wedge_{QR}Q''(FR'\wedge_{R'}Q'M),
\]
where $Q$ is a cofibrant replacement functor and $F$ is a fibrant replacement functor
on ring spectra, while $Q'$ is a cofibrant replacement functor on left
$R'$\nobreakdash-modules and $Q''$ is a cofibrant replacement functor on left $FR'$\nobreakdash-modules.
\end{lem}

\begin{proof}
If $R$ and $R'$ are weakly equivalent as ring spectra, there are ring morphisms
\[
R\longleftarrow QR \longrightarrow FR' \longleftarrow R'
\]
that are weak equivalences, since $QR$ is cofibrant and $FR'$ is fibrant.
Using these morphisms, we may view $FR'$ as a right $R'$\nobreakdash-module and $R$ as a right $QR$\nobreakdash-module.
By~\cite[Lemma~5.4.4]{HSS00}, smashing with a cofibrant
left module converts weak equivalences of right modules into weak equivalences of spectra.
Hence, the zig-zag of weak equivalences
\[
M \longleftarrow Q'M \longrightarrow FR'\wedge_{R'}Q'M \longleftarrow Q''(FR'\wedge_{R'}Q'M)
\longrightarrow R\wedge_{QR}Q''(FR'\wedge_{R'}Q'M)
\]
proves our claim.
\end{proof}

If $R$ and $R'$ are commutative, then induction and restriction
also yield a Quillen equivalence between the model categories of $R$\nobreakdash-algebras and $R'$\nobreakdash-algebras.

\begin{thm}
\label{preservmod}
Let $(L,\eta)$ be a homotopical localization on symmetric spectra. Let $R$ be
a ring spectrum and $M$ a left $R$\nobreakdash-module.
Suppose either that $L$ commutes with suspension or that $R$ is connective.
Then $\eta_M\colon M\longrightarrow LM$ is naturally weakly equivalent to a morphism
$\xi_M\colon DM\longrightarrow TM$ of left $R$\nobreakdash-modules 
where $DM\simeq M$ as $R$\nobreakdash-modules.
\end{thm}

\begin{proof}
Choose $P=\LMod_{\Ass}$ and consider the $P$\nobreakdash-algebra ${\bf X}=(R,M)$ ---which depends
functorially on $M$--- and the ideal $J=\{r\}$. We may assume that
${\bf X}$ is fibrant as a $P$\nobreakdash-algebra (otherwise, use a fibrant replacement
and Lemma~\ref{changeofrings}). Now
Theorem~\ref{javierslemma} implies that
$\eta_{\bf X}\colon {\bf X}\longrightarrow L{\bf X}$ is naturally
weakly equivalent to a map of $P$\nobreakdash-algebras $\xi_{\bf X}\colon D{\bf X}\longrightarrow T{\bf X}$
which depends functorially on $\bf X$ (hence on~$M$) and
where, in addition, $D{\bf X}\simeq {\bf X}$ as $P$\nobreakdash-algebras. By composing $\xi_{\bf X}$,
if necessary, with a cofibrant replacement of $D{\bf X}$ as a $P$\nobreakdash-algebra,
we may assume that $D{\bf X}$ is cofibrant.

Let us denote $D{\bf X}=(R',M')$ and $T{\bf X}=(R'',M'')$, and let $\mu\colon M'\longrightarrow M''$
be the morphism of $R'$\nobreakdash-modules induced by $\xi_{\bf X}$ on the second variable, which is weakly
equivalent to $\eta_M\colon M\longrightarrow LM$ as a map of spectra.

Now a change of rings is required.
Since $D{\bf X}$ is cofibrant and ${\bf X}$ is fibrant, there is a weak
equivalence of $P$\nobreakdash-algebras $f\colon D{\bf X}\longrightarrow {\bf X}$.
If we consider the inclusion $\alpha\colon \{r\}\longrightarrow \{r,m\}$, then
$R=\alpha^*{\bf X}$ and $R'=\alpha^*D{\bf X}$, and, since $\alpha^* P=\Ass$,
we can infer that the restriction of $f$ to the first component,
$\rho\colon R'\longrightarrow R$, is a weak equivalence of rings.
In this situation, by Lemma~\ref{changeofrings}, $\mu\colon M'\longrightarrow M''$ is naturally weakly
equivalent to a morphism of $R$\nobreakdash-modules $\xi_M\colon DM\longrightarrow TM$, where $DM\simeq R\wedge_{R'}M'$
and $TM\simeq R\wedge_{R'}M''$. Hence, $\eta_M$ is naturally weakly equivalent
to a morphism of $R$\nobreakdash-modules, as claimed.

In order to compare $DM$ with $M$, note that, since
$f\colon D{\bf X}\longrightarrow {\bf X}$ is a map of $P$\nobreakdash-algebras,
its second component can be viewed as a morphism of $R'$\nobreakdash-modules
\begin{equation}
\label{secondcomponent}
M'\longrightarrow \rho^* M,
\end{equation}
which is also a weak equivalence of spectra, hence a weak equivalence of $R'$\nobreakdash-modules.
Here $M$ is fibrant, and from the fact that $D{\bf X}$ is cofibrant it follows that
$M'$ is cofibrant as an $R'$\nobreakdash-module (since it has the left lifting property with respect
to all trivial fibrations of $R'$\nobreakdash-modules).
Since induction and restriction set up a Quillen equivalence, the adjoint map of (\ref{secondcomponent}),
\[
R\wedge_{R'}M'\longrightarrow M,
\]
is a weak equivalence of $R$\nobreakdash-modules. This shows that $DM\simeq M$ as $R$\nobreakdash-modules.
\end{proof}

Although this result was stated for left $R$\nobreakdash-modules,
it also holds of course for right $R$\nobreakdash-modules or $R$\nobreakdash-$S$\nobreakdash-bimodules, either 
by repeating the argument using the appropriate coloured operads,
or by replacing the ring spectrum $R$ by $R^{\rm op}$ and $R\wedge S^{\rm op}$, respectively.

\begin{thm}
\label{penultim}
Let $(L,\eta)$ be a homotopical localization on symmetric spectra and $f\colon R\longrightarrow S$ a morphism
of ring spectra. If either $L$ commutes with suspension or $R$, $LR$, $S$, and $LS$ are connective,
then $\eta_f\colon f\longrightarrow Lf$ is naturally weakly equivalent to a map $Df\longrightarrow Tf$
of ring morphisms, where $Df\simeq f$ as such. Hence, $Lf$ is naturally weakly
equivalent to a morphism of ring spectra.
\end{thm}

\begin{proof}
This follows from Theorem~\ref{javierslemma} by choosing $P=\Mor_{\Ass}$ and $J=\emptyset$.
\end{proof}

Note that this result implies Theorem~\ref{strictringspectra} 
by specializing $f$ to be the identity map of a ring spectrum~$R$.

\begin{thm}
\label{ultim}
Let $(L,\eta)$ be a homotopical localization on symmetric spectra and $g\colon M\longrightarrow N$
a morphism of left $R$\nobreakdash-modules, where $R$ is any ring spectrum.
If $L$ commutes with suspension or $R$ is connective, 
then $\eta_g\colon g\longrightarrow Lg$ is naturally weakly equivalent to a map $Dg\longrightarrow Tg$
of $R$\nobreakdash-module morphisms, where $Dg\simeq g$ as such. 
Hence, $Lg$ is naturally weakly equivalent to a morphism of $R$\nobreakdash-modules.
\end{thm}

\begin{proof}
Pick $P=\Mor_Q$ with $Q=\LMod_{\Ass}$ and $J=\{(0,r),(1,r)\}$.
If we denote by $\bf X$ the $P$\nobreakdash-algebra $(R,M)\longrightarrow (R,N)$ that
is the identity on the first variable and $g$ on the second variable, then
$\eta_{\bf X}\colon {\bf X}\longrightarrow L{\bf X}$ is a commutative diagram
\[
\xymatrix{
(R,M) \ar[d]_{({\rm id},g)} \ar[r]^-{({\rm id},\eta_M)} & (R,LM)\phantom{.} \ar[d]^{({\rm id},Lg)} \\
(R,N) \ar[r]^-{({\rm id},\eta_N)} & (R,LN).
}
\]
By Theorem~\ref{javierslemma}, this is naturally weakly equivalent to a
map $\xi_{\bf X}\colon D{\bf X}\longrightarrow T{\bf X}$ of $P$\nobreakdash-algebras, which we depict as
a commutative diagram
\[
\xymatrix{
(R',M') \ar[d]_{} \ar[r]^{} & (R'',M'')\phantom{.} \ar[d]^{(\rho,\nu)} \\
(T',N') \ar[r]^{} & (T'',N'').
}
\]
Here $\nu\colon M''\longrightarrow N''$ is therefore a morphism of $R'$\nobreakdash-modules.
From the fact that $D{\bf X}\simeq {\bf X}$ as $P$\nobreakdash-algebras it follows,
by restriction of colours as in the proof of Theorem~\ref{preservmod}, that $R'\simeq R$ as rings.
Thus Lemma~\ref{changeofrings} implies that $\nu$ is naturally weakly equivalent
to a morphism $Tg$ of $R$\nobreakdash-modules, and hence so is $Lg$.
\end{proof}

\section{Algebras over commutative ring spectra}
\label{algebras}
We finally discuss, as another application of our techniques,
the preservation of $R$\nobreakdash-algebras under homotopical localizations, where $R$ is a
commutative ring spectrum. For this, we first consider a convenient coloured operad.
In an arbitrary closed symmetric monoidal category~$\E$, choose $C=\{r, a\}$ and define
a $C$\nobreakdash-coloured operad $\mathcal{A}$ as follows:
\begin{equation}
\label{defRalg}
\mathcal{A}(c_1,\ldots,c_n;c)=\left\{
\begin{array}{l} \mbox{$0$ if $c=r$ and $c_k=a$ for some $k$,} \\[0.2cm]
\mbox{$I[\Sigma_n]/\sim$ otherwise,}
\end{array}
\right.
\end{equation}
where $I[\Sigma_n]$ denotes, as before, a coproduct of copies of the unit of $\E$ indexed
by the symmetric group $\Sigma_n$, and $\sim$ is the equivalence relation on $\Sigma_n$ defined in the
following way, similarly as in \cite[\S 9.3]{EM06}:
$\sigma\sim\sigma'$ if and only if, for all $i$ and $j$ such that $c_i=c_j=a$, the inequality
$\sigma(i)<\sigma(j)$ holds precisely when $\sigma'(i)<\sigma'(j)$ holds. For example,
\[
\mathcal{A}(r,\stackrel{(n)}{\ldots},r;r)=I
\quad
\mbox{and}
\quad
\mathcal{A}(a,\stackrel{(n)}{\ldots},a;a)=I[\Sigma_n].
\]
Thus, an algebra over $\mathcal{A}$ is a pair $(R,A)$ where $R$ is a commutative monoid
and $A$ is a (non\nobreakdash-commutative) monoid together with a ``central'' map $R\longrightarrow A$ given by
the structure map $\mathcal{A}(r;a)\otimes R\longrightarrow A$.
The ideals relative to $\mathcal{A}$ are $C$, $\{r\}$, and $\emptyset$.

The commutative algebras $A$ over a commutative monoid $R$
are the algebras over a $C$\nobreakdash-coloured operad defined as in (\ref{defRalg}),
but replacing $I[\Sigma_n]/\sim$ with~$I$. Note that the resulting coloured operad
precisely coincides with $\Mor_{\Com}$ after substituting $\{r,a\}$ by $\{0,1\}$.
Indeed, a commutative $R$\nobreakdash-algebra $A$ is nothing else but a morphism
$R\longrightarrow A$ of commutative monoids.

We now choose $\E$ to be the category of simplicial sets, acting on the
category $\M$ of symmetric spectra over simplicial sets with the positive stable model structure.
Then an algebra over $\mathcal{A}$ in $\M$ is a pair $(R,A)$ where $R$ is a commutative
ring spectrum and $A$ is an $R$\nobreakdash-algebra in the usual sense.

\begin{thm}
\label{Ralgebras}
Let $(L,\eta)$ be a homotopical localization on symmetric spectra.
Let $R$ be a commutative ring spectrum and let $A$ be an $R$\nobreakdash-algebra.
Suppose either that $L$ commutes with suspension or $R$, $A$, and $LA$ are connective.
Then $\eta_A\colon A\longrightarrow LA$ is naturally weakly equivalent to a morphism of
$R$\nobreakdash-algebras $\xi_A\colon DA\longrightarrow TA$ where $DA\simeq A$ as $R$\nobreakdash-algebras.
\end{thm}

\begin{proof}
Use the coloured operad $\mathcal{A}$ described above and apply Theorem~\ref{javierslemma}
localizing away from the ideal $J=\{r\}$ and with a change of rings as a final step.
The details are the same as those in the proof of Theorem~\ref{preservmod}.
\end{proof}

\begin{thm}
\label{mapsofRalgebras}
Let $(L,\eta)$ be a homotopical localization on symmetric spectra.
Let $R$ be a commutative ring spectrum and let
$g\colon A\longrightarrow B$ be a morphism of $R$\nobreakdash-algebras.
Suppose that $L$ commutes with suspension or $R$, $A$, $LA$, $B$, and $LB$ are connective.
Then $\eta_g\colon g\longrightarrow Lg$ is naturally weakly equivalent to a map $Dg\longrightarrow Tg$
of $R$\nobreakdash-algebra morphisms, where $Dg\simeq g$ as such.
Hence, $Lg$ is naturally weakly equivalent to a~morphism of $R$\nobreakdash-algebras.
\end{thm}

\begin{proof}
For this, use $\Mor_{\mathcal{A}}$ and localize away from $J=\{(0,r),(1,r)\}$.
\end{proof}

If we let $(L,\eta)$ be localization with respect to an arbitrary homology theory, then we
essentially recover Theorem VIII.2.1 in \cite{EKMM97}, stating that Bousfield localizations preserve
$R$\nobreakdash-algebras for every commutative ring spectrum $R$. A minor complication comes from the fact
that \cite{EKMM97} is written in terms of $S$\nobreakdash-modules instead of symmetric spectra.
A comparison can be made precise as follows. Let $\Psi$ set up a Quillen equivalence
from the category of $R$\nobreakdash-algebras in symmetric spectra over simplicial sets to the
category of $\Psi R$\nobreakdash-algebras in $S$\nobreakdash-modules, and let $\Phi$ be its right adjoint;
see \cite{Sch01} for further details about this adjoint pair.

For an $R$\nobreakdash-algebra $A$ in symmetric spectra, we denote by $\eta_A\colon A\longrightarrow L_EA$
its $E_*$\nobreakdash-localization, where $E_*$ is any homology theory.
By Theorem~\ref{Ralgebras}, $\eta_A$ is weakly equivalent to a morphism $\xi_A\colon D_EA\longrightarrow T_EA$
of $R$\nobreakdash-algebras.
Similarly, denote by $\lambda_{\Psi A}\colon {\Psi A}\longrightarrow (\Psi A)_E$
the $E_*$\nobreakdash-localization map in the category of $S$\nobreakdash-modules, and endow $(\Psi A)_E$ with the
$\Psi R$\nobreakdash-algebra structure of \cite[Theorem VIII.2.1]{EKMM97}.

\begin{prop}
Let $E_*$ be any homology theory.
Let $R$ be a commutative ring spectrum and let $A$ be an $R$\nobreakdash-algebra in symmetric spectra.
Then the $\Psi R$\nobreakdash-algebras $\Psi (T_EA)$ and $(\Psi A)_E$ are naturally weakly equivalent.
\end{prop}

\begin{proof}
The adjoint functors $(\Psi,\Phi)$ preserve and reflect $E_*$\nobreakdash-equivalences and $E_*$\nobreakdash-local spectra.
Therefore, $\Psi(\xi_A)$ is an $E_*$\nobreakdash-equivalence and $\Psi (T_EA)$ is $E$\nobreakdash-local. Hence, we infer from
\cite[Theorem VIII.2.1]{EKMM97} that $\lambda_{\Psi A}$ yields a natural map
\[
(\Psi A)_E\longrightarrow \Psi(T_EA)
\]
of $\Psi R$\nobreakdash-algebras which is a weak equivalence.
\end{proof}

In a different direction, we deduce that, for every \textit{connective} commutative ring spectrum $R$,
each \textit{connective} $R$\nobreakdash-algebra $A$ has a Postnikov tower consisting of $R$\nobreakdash-algebras.
Our argument is given below.
This result was proved by Lazarev in \cite[\S 8]{Laz01} with different methods,
extending previous results of Basterra and Kriz \cite[Theorem~8.1]{Bas99}. See also~\cite{DS06}.

\begin{prop}
Let $R$ be a connective, commutative, cofibrant ring spectrum, and let $A$ be a connective cofibrant $R$\nobreakdash-algebra.
For each $i\ge 0$ there are $R$\nobreakdash-algebra morphisms
\[
\xymatrix{
& A_{i+1} \ar[d]^{\tau_i} \\
A \ar[ru]^{\nu_{i+1}} \ar[r]^{\nu_{i}}& A_{i}
}
\]
such that the triangle commutes up to homotopy, $\nu_i$ induces isomorphisms on $\pi_n$
for $n\le i$, and $\pi_n(A_i)=0$ for $n > i$.
\end{prop}

\begin{proof}
For each $i\ge 0$, let $P_i$ denote localization with respect to
$f\colon \Sigma^{i+1}S \longrightarrow 0$ (where $S$ denotes the sphere spectrum)
in the category of symmetric spectra, and let $\eta_i$ be the corresponding coaugmentation.

From Theorem~\ref{Ralgebras} we infer that $\eta_i\colon A\longrightarrow P_iA$ is weakly equivalent
to a morphism of $R$\nobreakdash-algebras, which we denote by $\alpha_i\colon A'_i\longrightarrow A''_i$.
Let $A_i$ be a fibrant and cofibrant replacement of $A''_i$ in the model category of $R$\nobreakdash-algebras.
Thus $A_i\simeq P_iA$ as spectra.
Since $A$ is weakly equivalent to $A'_i$ as an $R$\nobreakdash-algebra,
we can consider the following composite of arrows in the homotopy category of $R$\nobreakdash-algebras:
\begin{equation}
\label{tobelifted}
\xymatrix{
A\ar@{.>}[r]^{\cong} &  A'_i\ar@{.>}[r]^{\alpha_i} & A''_i\ar@{.>}[r]^{\cong} & A_i.
}
\end{equation}
Since $A$ is cofibrant, there is a morphism of $R$\nobreakdash-algebras
$\nu_i\colon A\longrightarrow A_i$ lifting~(\ref{tobelifted}). Thus $\nu_i$ induces
isomorphisms on $\pi_n$ for $n\le i$, since $\eta_i$ does.

Now, since $\eta_{i}$ is a natural transformation, the following diagram commutes:
\[
\xymatrix{
A \ar[d]_{\eta_i} \ar[r]^-{\nu_{i+1}} & A_{i+1} \ar[d]^{\eta_i} \\
P_iA \ar[r]^-{P_i\nu_{i+1}} & P_{i}A_{i+1}.
}
\]
In this diagram, the lower horizontal arrow is a weak equivalence of spectra,
and, by Theorem~\ref{mapsofRalgebras}, it is weakly equivalent to a morphism
$\beta_i\colon B'_i\longrightarrow B''_i$ of $R$\nobreakdash-algebras, which is therefore
a weak equivalence of $R$\nobreakdash-algebras.
Likewise, by Theorem~\ref{Ralgebras},
the right\nobreakdash-hand vertical map is weakly equivalent to a morphism
$\gamma_i\colon C'_i\longrightarrow C''_i$ of $R$\nobreakdash-algebras, where
in addition $C'_i\simeq A_{i+1}$ as $R$\nobreakdash-algebras.

It follows from part (iii) of Theorem~\ref{javierslemma}
that each of these rectification steps is in fact a weak equivalence of
$(\Mor_{\mathcal{A}})_{\infty}$\nobreakdash-algebras.
(Here we have used the assumption that $R$ is cofibrant.)
Thus, by restriction of colours,
they induce weak equivalences of $\mathcal{A}_{\infty}$\nobreakdash-algebras on their domains and targets.
Hence, $(R,C''_i)$ and $(R,B''_i)$ are weakly equivalent as $\mathcal{A}_{\infty}$\nobreakdash-algebras.
By Lemma~\ref{rectifying}, they are in fact
weakly equivalent as $\mathcal{A}$\nobreakdash-algebras, meaning that $C''_i\simeq B''_i$
as $R$\nobreakdash-algebras. Similarly, $B'_i\simeq A''_i$ as $R$\nobreakdash-algebras.
Since $\beta_i$ is invertible in the homotopy category of $R$\nobreakdash-algebras,
we may consider the composite
\[
\xymatrix{
A_{i+1}\ar@{.>}[r]^{\cong} &  C'_i\ar@{.>}[r]^{\gamma_i} & C''_i\ar@{.>}[r]^{\cong} &
B''_i\ar@{.>}[r]^{\beta_i^{-1}} & B'_i\ar@{.>}[r]^{\cong} & A''_i\ar@{.>}[r]^{\cong} & A_i,
}
\]
which can be lifted to a map of $R$\nobreakdash-algebras $\tau_i\colon A_{i+1}\longrightarrow A_i$.
By construction, $\tau_i\circ\nu_{i+1}$ coincides with $\nu_i$ in the homotopy category of $R$\nobreakdash-algebras,
so $\tau_i\circ\nu_{i+1}\simeq\nu_i$ as $R$\nobreakdash-algebra morphisms, as claimed.
\end{proof}


\begin{thebibliography}{99}
\bibitem[AR94]{AR94} J. Ad\'amek \and J. Rosick\'y, \textit{Locally Presentable and Accessible Categories}, London Math. Soc.
Lecture Note Ser., vol. 189, Cambridge University Press, Cambridge, 1994.
\bibitem[Bad02]{Bad02} B. Badzioch, Algebraic theories in homotopy theory, \textit{Ann. of Math.} \textbf{155} (2002), no.\ 3, 895\nobreakdash--913.
\bibitem[Bas99]{Bas99} M. Basterra, Andr\'e\nobreakdash--Quillen cohomology of commutative $S$\nobreakdash-algebras,
\textit{J. Pure Appl. Algebra} {\bf 144} (1999), 111\nobreakdash--143.
\bibitem[BM03]{BM03} C. Berger \and I. Moerdijk, Axiomatic homotopy theory for operads, \textit{Comment. Math. Helv.} \textbf{78} (2003), no.\ 4, 805\nobreakdash--831.
\bibitem[BM06]{BM06} C. Berger \and I. Moerdijk, The Boardman\nobreakdash--Vogt resolution of operads in monoidal model categories,
\textit{Topology} \textbf{45} (2006), no.\ 5, 807\nobreakdash--849.
\bibitem[BM07]{BM07} C. Berger \and I. Moerdijk, Resolution of coloured operads and rectification of homotopy
algebras, in: \textit{Categories in Algebra, Geometry and Mathematical Physics (Street Festschrift)},
Contemp. Math., vol. 431, Amer. Math. Soc., Providence, 2007, 31\nobreakdash--58.
\bibitem[Ber06]{Ber06} J. E. Bergner, Rigidification of algebras over multi\nobreakdash-sorted theories, \textit{Algebr.
Geom. Topol.} {\bf 6} (2006) 1925\nobreakdash--1955.
\bibitem[BV73]{BV73} J. M. Boardman \and R. M. Vogt, \textit{Homotopy Invariant Algebraic Structures on Topological Spaces}, Lecture
Notes in Math., vol. 347, Springer\nobreakdash-Verlag, Berlin, Heidelberg, 1973.
\bibitem[Bou94]{Bou94} A. K. Bousfield, Localization and periodicity in unstable homotopy theory, \textit{J. Amer. Math. Soc.} \textbf{7} (1994), 831\nobreakdash--873.
\bibitem[Bou96]{Bou96} A. K. Bousfield,  Unstable localization and
periodicity, in: \textit{Algebraic Topology: New Trends in Localization and Periodicity (Sant Feliu de Gu\'{\i}xols, 1994)}, Progress in Math., vol. 136, Birkh\"auser, Basel, 1996, 33\nobreakdash--50.
\bibitem[Bou99]{Bou99} A. K. Bousfield, On $K(n)$\nobreakdash-equivalences of spaces, in: \textit{Homotopy Invariant
Algebraic Structures (Baltimore, 1998)}, Contemp. Math., vol. 239, Amer. Math. Soc., Providence, 1999, 85\nobreakdash--89.
\bibitem[Cas00]{Cas00} C. Casacuberta, On structures preserved by idempotent transformations of groups and homotopy types,
in: \textit{Crystallographic Groups and their Generalizations (Kortrijk, 1999)}, Contemp. Math., vol. 262, Amer.
Math. Soc., Providence, 2000, 39\nobreakdash--68.
\bibitem[CC06]{CC06} C. Casacuberta \and B. Chorny, The orthogonal subcategory problem in homotopy theory,
in: \textit{An Alpine Anthology of Homotopy Theory (Arolla, 2004)}, Contemp. Math., vol. 399,
Amer. Math. Soc., Providence, 2006, 41\nobreakdash--53.
\bibitem[CG05]{CG05} C. Casacuberta \and J. J. Guti\'errez, Homotopical localization of module spectra, \textit{Trans. Amer. Math. Soc.}
\textbf{357} (2005), no. 7, 2753\nobreakdash--2770.
\bibitem[DS06]{DS06} D. Dugger and B. Shipley, Postnikov extensions of ring spectra, \textit{Algebr. Geom. Topol.}
\textbf{6} (2006), 1785\nobreakdash--1829.
\bibitem[EKMM97]{EKMM97} A. D. Elmendorf, I. Kriz, M. A. Mandell, \and J. P. May, \textit{Rings, Modules, and Algebras in Stable Homotopy Theory},
Math. Surveys and Monographs, vol. 47, Amer. Math. Soc., Providence, 1997.
\bibitem[EM06]{EM06} A. D. Elmendorf \and M. A. Mandell, Rings, modules, and algebras in infinite loop space theory,
\textit{Adv. Math.} \textbf{205} (2006), no.\ 1, 163\nobreakdash--228.
\bibitem[Far96]{Far96} E. Dror Farjoun, \textit{Cellular Spaces, Null Spaces and Homotopy Theory}, Lecture Notes in Math., vol. 1622,
Springer\nobreakdash-Verlag, Berlin, Heidelberg, 1996.
\bibitem[GH04]{GH04} P. G. Goerss \and M. Hopkins, Moduli spaces of commutative ring spectra, in: \textit{Structured Ring Spectra
(Glasgow, 2002)}, London Math. Soc. Lecture Note Ser., vol. 315, Cambridge University Press, Cambridge, 2004.
\bibitem[GJ99]{GJ99} P. G. Goerss \and J. F. Jardine, \textit{Simplicial Homotopy Theory}, Progress in Math., vol. 174, Birkh\"auser, Basel, 1999.
\bibitem[Hir03]{Hir03} P. S. Hirschhorn, \textit{Model Categories and Their Localizations}, Math. Surveys and Monographs, vol.~99,
Amer. Math. Soc., Providence, 2003.
\bibitem[Hov99]{Hov99} M. Hovey, \textit{Model Categories}, Math. Surveys and Monographs, vol.~63, Amer. Math. Soc., Providence, 1999.
\bibitem[HSS00]{HSS00} M. Hovey, B. Shipley, \and J. H. Smith, Symmetric spectra, \textit{J. Amer. Math. Soc.} \textbf{13} (2000), no. 1,
149\nobreakdash--208.
\bibitem[Lam69]{Lam69} J. Lambek, Deductive systems and categories, II. Standard constructions and closed categories,
in: \textit{Category Theory, Homology Theory and Their Applications, I (Seattle, 1968)}, Lecture Notes in Math., vol. 68,
Springer\nobreakdash-Verlag, Berlin, Heidelberg, 1969, 76\nobreakdash--122.
\bibitem[Laz01]{Laz01} A. Lazarev, Homotopy theory of $A_{\infty}$ ring spectra and applications to $MU$\nobreakdash-modules,
\textit{$K$\nobreakdash-Theory} \textbf{24} (2001), 243\nobreakdash--281.
\bibitem[MMSS01]{MMSS01} M. A. Mandell, J. P. May, S. Schwede, \and B. Shipley, Model categories of diagram spectra,
\textit{Proc. London Math. Soc.} \textbf{82} (2001), no. 2, 441\nobreakdash--512.
\bibitem[Mar04]{Mar04} M. Markl, Homotopy algebras are homotopy algebras, \textit{Forum Math.} \textbf{16} (2004), no. 1, 129\nobreakdash--160.
\bibitem[MSS02]{MSS02} M. Markl, S. Shnider, \and J. Stasheff, \textit{Operads in Algebra, Topology and Physics}, Math. Surveys and Monographs,
vol.~96, Amer. Math. Soc., Providence, 2002.
\bibitem[May74]{May74} J. P. May, $E_{\infty}$ spaces, group completions, and permutative categories, in: \textit{New Developments
in Topology (Oxford, 1972)}, London Math. Soc. Lecture Note Ser., vol. 11, Cambridge University Press, Cambridge, 1974, 61\nobreakdash--93.
\bibitem[Qui67]{Qui67} D. G. Quillen, \textit{Homotopical Algebra}, Lecture Notes in Math., vol.~43, Springer\nobreakdash-Verlag, Berlin, Heidelberg, 1967.
\bibitem[Rud98]{Rud98} Y. B. Rudyak, \textit{On Thom spectra, Orientability, and Cobordism},
Springer Monographs in Mathematics, Springer\nobreakdash-Verlag, Berlin, 1998.
\bibitem[SV88]{SV88} R. Schw\"anzl \and R. M. Vogt, $E_{\infty}$\nobreakdash-spaces and injective $\Gamma$\nobreakdash-spaces,
\textit{Manuscripta Math.} {\bf 61} (1988), no. 2, 203\nobreakdash--214.
\bibitem[Sch01]{Sch01} S. Schwede, $S$\nobreakdash-modules and symmetric spectra,
\textit{Math. Ann.} {\bf 319} (2001), 517\nobreakdash--532.
\bibitem[SS00]{SS00} S. Schwede \and B. Shipley, Algebras and modules in monoidal model categories,
\textit{Proc. London Math. Soc.} \textbf{80} (2000), 491\nobreakdash--511.
\bibitem[Shi04]{Shi04} B. Shipley, A convenient model category for commutative ring spectra.
in: \textit{Homotopy Theory: Relations with Algebraic Geometry, Group Cohomology, and Algebraic K\nobreakdash-Theory}, Contemp.
Math., vol. 346, Amer. Math. Soc., Providence, 2004, 473\nobreakdash--484.
\bibitem[Sta63]{Sta63} J. Stasheff, Homotopy associativity of $H$\nobreakdash-spaces, \textit{Trans. Amer. Math. Soc.}
{\bf 108} (1963), 275\nobreakdash--312.
\bibitem[Str72]{Str72} A. Str{\o}m, The homotopy category is a homotopy category, \textit{Arch. Math. (Basel)}
{\bf 23} (1972), 435\nobreakdash--441.
\bibitem[Vog71]{Vog71} R. M. Vogt, Convenient categories of topological spaces for homotopy theory,
\textit{Arch. Math. (Basel)} {\bf 22} (1971), 545\nobreakdash--555.
\bibitem[Vog03]{Vog03} R. M. Vogt, Cofibrant operads and universal $E_{\infty}$ operads,
\textit{Topology Appl.} {\bf 133} (2003), no. 1, 69\nobreakdash--87.
\end{thebibliography}
\end{document}